\documentclass[12pt]{gtart}
\usepackage{amsmath}
\usepackage{amssymb}
\usepackage{amsthm}
\usepackage{color}
\usepackage{pinlabel}

\newtheorem{theorem}{Theorem}
\newtheorem{lemma}[theorem]{Lemma}

\newtheorem{corollary}[theorem]{Corollary}

\theoremstyle{definition}
\newtheorem{definition}[theorem]{Definition}

\newtheorem{remark}[theorem]{Remark}

\def\Z{\mathbb Z}

\def\R{\mathbb R}
\def\C{\mathbb C}

\def\p{\partial}

\newcommand{\into}{\ensuremath{\hookrightarrow}}

\newcommand{\Crit}{\mathop{\rm Crit}\nolimits}

\begin{document}

\title{Trisecting $4$--manifolds}
\authors{ David T. Gay, Robion Kirby
\footnote{This work was partially supported by a grant from the Simons
  Foundation (\#210381 to David Gay). The first author was partially supported by NSF grant DMS-1207721. The second author was partially
  supported by NSF grant DMS-0838703.}}  
\address{Euclid Lab, 160 Milledge Terrace, Athens, GA 30606\\ Department of Mathematics,University
  of Georgia, Athens, GA 30602} \secondaddress{University of
  California, Berkeley, CA 94720} \email{d.gay@euclidlab.org}
\secondemail{kirby@math.berkeley.edu}
\begin{abstract}

We show that any smooth, closed, oriented, connected $4$--manifold
can be trisected into three copies of $\natural^k (S^1 \times B^3)$, intersecting pairwise in $3$--dimensional handlebodies, with triple intersection a closed $2$--dimensional surface.  Such a trisection is unique up to a natural stabilization operation.  This is analogous to the
existence, and uniqueness up to stabilization, of Heegaard splittings of $3$--manifolds. A trisection of a $4$--manifold $X$ arises from a Morse $2$--function $G:X \to B^2$ and the obvious trisection of $B^2$, in much the same way that a Heegaard splitting of a $3$--manifold $Y$ arises from a Morse function $g : Y \to B^1$ and the obvious bisection of $B^1$.

\end{abstract}

\maketitle

\section{Introduction}

Consider first the $3$--dimensional case of an oriented, connected,
closed $3$--manifold $Y^3$.  From a Morse function
$f:Y \to [0,3]$ with only one critical point of index zero and one of
index three, and all critical points of index $i$ mapping to $i$, we
see that $f^{-1} ([0,3/2])$ and $f^{-1}([3/2,3])$ are solid handlebodies,
$\natural^g (S^1 \times B^2)$ .

For uniqueness, we use Cerf theory \cite{cerf} to get a homotopy
$f_t:Y \to [0,3]$ between $f_0$ and $f_1$ (each giving Heegaard
splittings) where this homotopy introduces no new critical points of
index zero or three.  There are births and deaths of
canceling pairs of index one and two critical points, but these stabilize the Heegaard
splittings by connected summing with the standard genus one splitting
of $S^3$.  The homotopy $f_t$ can be chosen to keep the index one
critical values below $3/2$ and the index two above.  Then handle
slides between $1$--handles, or $2$--handles, take one Heegaard
splitting to the other.  (This is a now well known Cerf theoretic
proof of the Reidemeister-Singer theorem \cite{reide}\cite{singer}
which was originally proved combinatorially e.g. \cite{saveliev}.)

Recall that a Heegaard {\em diagram} for a Heegaard splitting is a triple $(F_g, \alpha, \beta)$ where $F_g$ is the Heegaard surface, and each of $\alpha$ and $\beta$ is a $g$--tuple of simple closed curves in $F_g$ which bounds a basis of compressing disks in each of the two handlebodies. Thus every $3$--manifold is described by a Heegaard diagram, and two Heegaard diagrams describe diffeomorphic $3$--manifolds if and only if they are related by stabilization, handle slides, and diffeomorphisms of $F_g$.

We now set up an analogous story in dimension four: Let $Z_k = \natural^k (S^1 \times B^3)$
with $Y_k = \partial Z_k = \sharp^k (S^1 \times S^2)$. Given an
integer $g \geq k$, let $Y_k = Y_{k,g}^+ \cup Y_{k,g}^-$ be the
standard genus $g$ Heegaard splitting of $Y_k$ obtained by stabilizing
the standard genus $k$ Heegaard splitting $g-k$ times. 
\begin{definition} \label{D:Trisection}
 Given integers $0 \leq k \leq g$, a {\em $(g,k)$--trisection} (see Figure~\ref{F:BasicTrisection}) of a closed, connected, oriented $4$--manifold $X$ is a decomposition of $X$ into three submanifolds $X = X_1 \cup X_2 \cup X_3$ satisfying the following properties:
\begin{enumerate}
 \item For each $i = 1,2,3$, there is a diffeomorphism $\phi_i \co X_i \to Z_k$.
 \item For each $i = 1,2,3$, taking indices mod $3$, $\phi_i(X_i \cap
   X_{i+1}) = Y_{k,g}^-$ and $\phi_i(X_i \cap X_{i-1}) = Y_{k,g}^+$.
\end{enumerate}
\end{definition}

\begin{remark}
 Note that the triple intersection $X_1 \cap X_2 \cap X_3$ is a surface of genus $g$ and that $\chi(X)=2+g-3k$. Thus $k$ is determined by $X$ and $g$, and for this reason we will often refer to a $(g,k)$--trisection of $X$ simply as a {\em genus $g$ trisection} of $X$. Also note that, for a fixed $X$, different trisections thus have the same genera mod $3$.
\end{remark}

Given a $(g,k)$--trisection $X=X_1 \cup X_2 \cup X_3$, consider the handlebodies $H_{ij} = X_i \cap X_j$ and the central genus $g$ surface $F_g = X_1 \cap X_2 \cap X_3 = \partial H_{ij}$. A choice of a system of $g$ compressing disks on $F_g$ for each of the three handlebodies gives three collections of $g$ curves: $\alpha = (\alpha_1, \ldots, \alpha_g)$,  $\beta = (\beta_1, \ldots, \beta_g)$ and $\gamma = (\gamma_1, \ldots, \gamma_g)$, such that compressing along $\alpha$ gives $H_{12}$, compressing along $\beta$ gives $H_{23}$ and compressing along $\gamma$ gives $H_{31}$. Furthermore, each pair $(\alpha,\beta)$, $(\beta,\gamma)$ and $(\gamma,\alpha)$ is a Heegaard diagram for $\sharp^k (S^1 \times S^2)$.

\begin{definition} \label{D:TrisectionDiagram}
 A {\em $(g,k)$--trisection diagram} is a $4$--tuple $(F_g, \alpha, \beta, \gamma)$ such that each triple $(F_g,\alpha,\beta)$, $(F_g,\beta,\gamma)$ and $(F_g,\gamma,\alpha)$ is a genus $g$ Heegaard diagram for $\sharp^k (S^1 \times S^2)$. The $4$--manifold determined in the obvious way by this trisection diagram will be denoted $X(F_g, \alpha, \beta, \gamma)$.
\end{definition}

\begin{theorem}[Existence] \label{T:Existence}
 Every closed, connected, oriented $4$--manifold $X$ has a $(g,k)$--trisection for some $0 \leq k \leq g$. Furthermore, $g$ and $k$ are such that $X$ has a handlebody decomposition with $1$ $0$--handle, $k$ $1$--handles, $g-k$ $2$--handles, $k$ $3$--handles and $1$ $4$--handle.
\end{theorem}

\begin{figure}
\labellist
\tiny\hair 2pt
\pinlabel $X_1$ at 35 85
\pinlabel $X_2$ at 65 33
\pinlabel $X_3$ at 95 85

\endlabellist
\centering
 \includegraphics[width=1in]{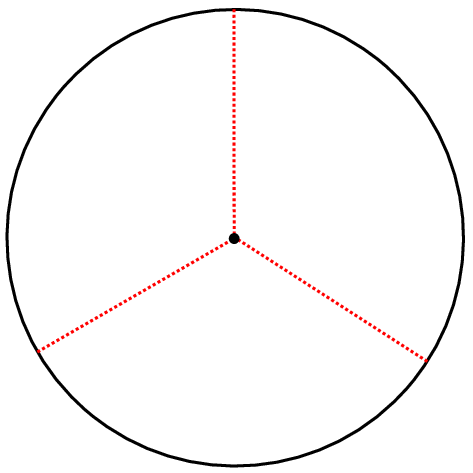}
 \caption{\label{F:BasicTrisection} How the pieces of a trisection fit together.}
\end{figure}

\begin{remark}
 There are two trivial consequences of the handle decomposition mentioned in the theorem which are worth noting:
\begin{enumerate}
 \item If $k=0$, i.e. $X_1$, $X_2$ and $X_3$ are each $4$--balls, then $X$ has no $1$-- or $3$--handles, and is thus simply-connected.
 \item If $g=k$, then $X$ has no $2$--handles, so $X \cong \sharp^k S^1 \times S^3$. 
\end{enumerate}
\end{remark}

The following is immediate:
\begin{corollary}
 Every closed $4$--manifold is diffeomorphic to $X(F_g,\alpha,\beta,\gamma)$ for some trisection diagram $(F_g,\alpha,\beta,\gamma)$.
\end{corollary}

\begin{remark}
 Readers familiar with the Heegaard triples used~\cite{OzsSzTriangles} to define the Heegaard-Floer $4$--manifold invariants will see that a trisection diagram is a special type of Heegaard triple and may suspect that this corollary follows fairly quickly from the Heegaard triple techniques in~\cite{OzsSzTriangles}. In all fairness this is probably true; we will present two proofs of Theorem~\ref{T:Existence}, one of which tells the story of how we discovered the result using Morse $2$--functions, while the other is more in the spirit of~\cite{OzsSzTriangles}, directly using ordinary handle decompositions. In some sense, then, our existence result can be thought of as a particularly nice packaging of the topological setup for~\cite{OzsSzTriangles}.
\end{remark}

Exactly as with Heegaard splittings in dimension $3$, our uniqueness result for trisections of $4$--manifolds is uniqueness up to a stabilization operation, which we now define. The idea is illustrated in Figure~\ref{F:3DTrisectionStabilized}, in dimension $3$.
\begin{figure}
\labellist
\tiny\hair 2pt
\pinlabel $X_1$ at 82 227
\pinlabel $X_2$ at 223 90
\pinlabel $X_3$ at 251 183
\pinlabel $H_{31}$ at 210 223
\pinlabel $H_{12}$ at 78 161
\pinlabel $H_{23}$ at 243 125
\endlabellist
\centering
 \includegraphics[width=2.5in]{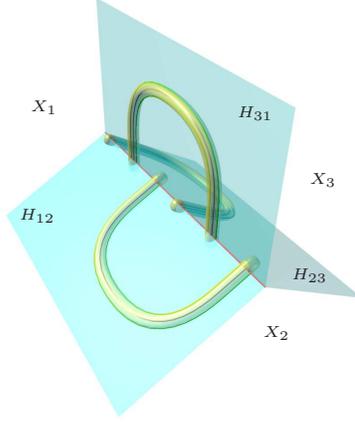}
 \caption{\label{F:3DTrisectionStabilized} Stabilizing a trisection in dimension $3$.}
\end{figure}

\begin{definition}[Stabilization] \label{D:Stabilization}
 Given a $4$--manifold $X$ with a trisection $(X_1,X_2,X_3)$, we construct a new trisection $(X'_1,X'_2,X'_3)$, as follows: For each $i,j \in \{1,2,3\}$, let $H_{ij}$ be the handlebody $X_i \cap X_j$, with boundary $F=X_1 \cap X_2 \cap X_3$. Let $a_{ij}$ be a properly embedded boundary parallel arc in each $H_{ij}$, such that the end points of $a_{12}$, $a_{23}$ and $a_{31}$ are disjoint in $F$. Let $N_{ij}$ be a closed $4$--dimensional regular neighborhood of $a_{ij}$ in $X$ (thus diffeomorphic to $B^4$), with $N_{12}$, $N_{23}$ and $N_{31}$ disjoint. Then we define:
\begin{itemize}
 \item $X'_1 = (X_1 \cup N_{23}) \setminus (\mathring{N}_{31} \cup \mathring{N}_{12})$
 \item $X'_2 = (X_2 \cup N_{31}) \setminus (\mathring{N}_{12} \cup \mathring{N}_{23})$
 \item $X'_3 = (X_3 \cup N_{12}) \setminus (\mathring{N}_{23} \cup \mathring{N}_{31})$
\end{itemize}
The operation of replacing $(X_1,X_2,X_3)$ with $(X'_1,X'_2,X'_3)$ is called {\em stabilization}.
\end{definition}

Since any two boundary parallel arcs in a handlebody are isotopic, it is clear that this operation does not depend on the choice of arcs or neighborhoods. 

In terms of trisection diagrams we have: 
\begin{definition}
 Given a trisection diagram $(F_g,\alpha,\beta,\gamma)$, the trisection diagram $(F'_{g'}=F_{g+3},\alpha',\beta',\gamma')$ obtained by connected summing $(F_g,\alpha,\beta,\gamma)$ with the diagram in Figure~\ref{F:S4nontrivial} is called the {\em stabilization} of $(F_g,\alpha,\beta,\gamma)$.
\end{definition}
\begin{figure}
\labellist
\tiny\hair 2pt
\endlabellist
\centering
 \includegraphics{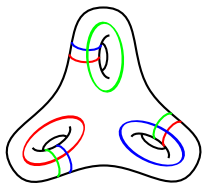}
 \caption{\label{F:S4nontrivial} Stabilizing a trisection diagram means connected summing with this diagram. By itself, this describes the simplest nontrivial trisection of $S^4$, of genus $3$. The red curves are the $\alpha$ curves, the blues are the $\beta$'s and the greens are the $\gamma$'s.}
\end{figure}

We prove the following fact at the beginning of Section~\ref{S:Uniqueness}:

\begin{lemma} \label{L:Stabilization}
 If $(X_1,X_2,X_3)$ is a trisection of $X^4$, with genus $g$ and diagram $(F_g,\alpha,\beta,\gamma')$, and $(X_1',X_2',X_3')$ is a stabilization of $(X_1,X_2,X_3)$, then $(X_1',X_2',X_3')$ is also a trisection of $X$, with genus $g'=g+3$ and diagram $(F_{g'},\alpha',\beta',\gamma')$, the stabilization of $(F_g,\alpha,\beta,\gamma)$.
\end{lemma}

The reader may find Figure~\ref{F:3DTrisectionStabilized} useful in proving this lemma before reading our proof.

\begin{theorem}[Uniqueness] \label{T:Uniqueness}
Given two trisections of $X$, $(X_1, X_2, X_3)$ and $(X_1', X_2',
X_3')$, then after stabilizing each trisection some number of times, there is a
diffeomorphism $h:X \to X$ isotopic to the identity with the property that $h(X_i) = X_i'$ for each $i$. In particular, $h(X_i \cap X_j)= X_i' \cap X_j',$
for $i \neq j$ in $\{1,2,3\}$, and $h(X_1 \cap X_2 \cap X_3) = h(X'_1 \cap X'_2 \cap X'_3)$.
\end{theorem}

\begin{corollary} \label{C:UniquenessDiagrams}
 Given trisection diagrams $(F_g,\alpha,\beta,\gamma)$ and $(F_{g'},\alpha',\beta',\gamma')$, the corresponding $4$--manifolds $X(F_g,\alpha,\beta,\gamma)$ and $X(F_{g'},\alpha',\beta',\gamma')$ are diffeomorphic if and only if $(F_g,\alpha,\beta,\gamma)$ and $(F_{g'},\alpha',\beta',\gamma')$ are related by stabilization, handle slides, and diffeomorphism. ``Handle slides'' are slides of $\alpha$'s over $\alpha$'s, $\beta$'s over $\beta$'s and $\gamma$'s over $\gamma$'s.
\end{corollary}

\begin{proof}
 Any two handle decompositions of a fixed genus $g$ handlebody, each with one $0$--handle and $g$ $1$--handles, are related by handle slides; this is proved in~\cite{Johannson}.
\end{proof}

\section{Discussion and examples} \label{S:Examples}

We begin with a few explicit examples of trisections and corresponding trisection diagrams.
\begin{itemize}
 \item $S^4 \subset \C \times \R^3$ can be explicitly divided into three pieces $X_j = \{(r e^{i \theta}, x_3, x_4, x_5) \mid 2\pi j/3 \leq \theta \leq 2\pi(j+1)/3 \}$, giving a genus $0$ trisection of $S^4$. The diagram is $S^2$ with no curves.
 \item Stabilizing the genus $0$ trisection of $S^4$ gives a genus $3$ trisection, with trisection diagram shown in Figure~\ref{F:S4nontrivial}. Since it is not known if the mapping class group of $S^4$ is trivial, we cannot say that the diagram determines the trisection up to isotopy, but the original description of stabilization of trisections (as opposed to stabilization of trisection diagrams) does determine this trisection up to isotopy, and thus we call this the {\em standard genus $3$ trisection of $S^4$}.
 \item There is an obvious {\em connected sum} operation on trisected $4$--manifolds, obtained by removing standardly trisected balls from each manifold and gluing along the boundary spheres so as to match the trisections. Stabilization can then also be defined as performing a connected sum with $S^4$ with its standard genus $3$ trisection.
 \item The standard toric picture of $\C P^2$ as a right triangle gives a natural trisection into three pieces $X_1, X_2, X_3$ as the inverse images under the moment map of the three pieces of the right triangle shown in Figure~\ref{F:ToricCP2}. These pieces are diffeomorphic to $B^4$ but they intersect along solid tori all meeting along a central fiber diffeomorphic to $T^2$, so that this is a genus $1$ trisection of $\C P^2$. The trisection diagram shows a $(1,0)$--, a $(0,1)$-- and a $(1,1)$-- curve; this is because the normals to the edges of the moment polytope tell us the direction in the torus which collapses along that edge. Alternatively, this trisection can be seen simply as the $0$--handle, $2$--handle and $4$--handle in the standard handle decomposition of $\C P^2$, and the $+1$ framing on the $2$--handle can be seen in the $(1,1)$--curve.
\begin{figure}
\labellist
\tiny\hair 2pt
\pinlabel $X_1$ at 25 55
\pinlabel $X_2$ at 25 25
\pinlabel $X_3$ at 55 25
\pinlabel $\alpha$ [r] at 160 24
\pinlabel $\beta$ [br] at 204 45
\pinlabel $\gamma$ [l] at 215 52
\endlabellist
\centering
 \includegraphics{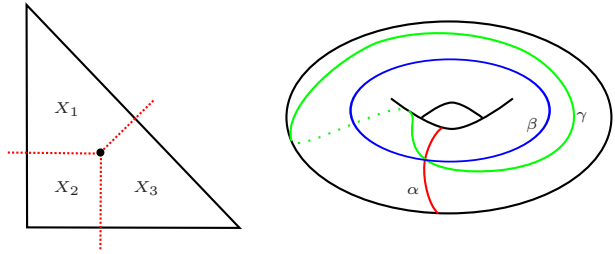}
 \caption{\label{F:ToricCP2} Trisection of $\C P^2$.}
\end{figure}
 \item Reversing the orientation of the central surface in a trisection diagram reverses the orientation of the $4$--manifold; i.e. $X(F_g,\alpha,\beta,\gamma) = -X(-F_g,\alpha,\beta,\gamma)$. Thus $\overline{\C P^2}$ has a genus $1$ trisection, with trisection diagram given by a $(1,0)$--, $(0,1)$-- and $(1,-1)$--curve.
 \item Looking at the standard toric picture of $S^2 \times S^2$ as a square also leads to a natural trisection of $S^2 \times S^2$ as follows: We divide the square into four regions labelled $X_1$, $X_{2a}$, $X_{2b}$ and $X_3$ as indicated in Figure~\ref{F:ToricS2XS2}, and label the inverse images of these regions in $S^2 \times S^2$ with the same labels.
\begin{figure}
\labellist
\tiny\hair 2pt
\pinlabel $X_1$ at 44 44
\pinlabel $X_{2a}$ at 110 35
\pinlabel $X_{2b}$ at 30 110
\pinlabel $X_3$ at 104 104
\endlabellist
\centering
 \includegraphics{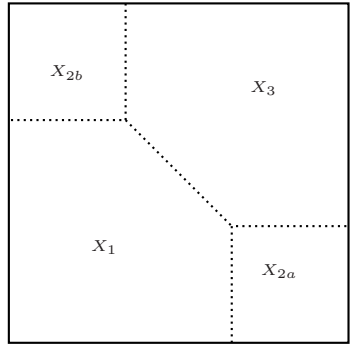}
 \caption{\label{F:ToricS2XS2} Trisection of $S^2 \times S^2$.}
\end{figure}
Each of $X_1$, $X_{2a}$, $X_{2b}$ and $X_3$ is a $4$--ball, and in fact they give the standard handle decomposition of $S^2 \times S^2$, with $X_1$ being the $0$--handle, $X_{2a}$ and $X_{2b}$ being the $2$--handles and $X_3$ being the $4$--handle. Note that $X_1 \cap X_3$ is $T^2 \times [0,1]$, with $T^2 \times \{0\}$ being $X_1 \cap X_3 \cap X_{2a}$ and $T^2 \times \{1\}$ being $X_1 \cap X_3 \cap X_{2b}$. Let $p$ be a point in $T^2$ and let $a$ be the arc $\{p\} \times [0,1] \subset X_1 \cap X_3$. Remove a tubular neighborhood of this arc from $X_1$ and $X_3$ and add it as a tube joining $X_{2a}$ to $X_{2b}$. The union of $X_{2a}$ and $X_{2b}$ with this tube is the $X_2$ of our trisection, and the new $X_1$ and $X_3$ are the results of removing the tube from the original $X_1$ and $X_3$. A little thought shows that this is a trisection with $k=0$ (each piece is a $4$--ball) and $g=2$. (Thanks to Bob Edwards for giving us the initial picture that led to this description.)
\item It may not be entirely obvious how to draw the trisection diagram for the above trisection of $S^2 \times S^2$. However, it is not hard to draw a genus $2$ trisection diagram from scratch that does give $S^2 \times S^2$. In Figure~\ref{F:VariousGenus2} we show this diagram, as well as a diagram for $S^2 \widetilde{\times} S^2$ and a diagram for $S^1 \times S^3$. We leave it to the reader to see how to relate these diagrams to the standard handle diagrams for these $4$--manifolds. It is also an illuminating exercise, knowing that $S^2 \widetilde{\times} S^2 \cong \C P^2 \sharp \overline{\C P^2}$, to verify Corollary~\ref{C:UniquenessDiagrams} in this case. The earlier discussion of connected sums and of $\pm \C P^2$ gives a trisection diagram for $\C P^2 \sharp \overline{\C P^2}$ and one checks that this is equivalent to that in Figure~\ref{F:VariousGenus2} for $S^2 \widetilde{\times} S^2$ via handle slides and diffeomorphism of $F_g$. (It turns out that in this case we do not need stabilization.)
\begin{figure}[h]
\labellist
\tiny\hair 2pt
\pinlabel {$S^2 \times S^2$} [b] at 54 49
\pinlabel {$S^2 \widetilde{\times} S^2$} [b] at 170 49
\pinlabel {$S^1 \times S^3$} [b] at 272 50
\endlabellist
\centering
 \includegraphics{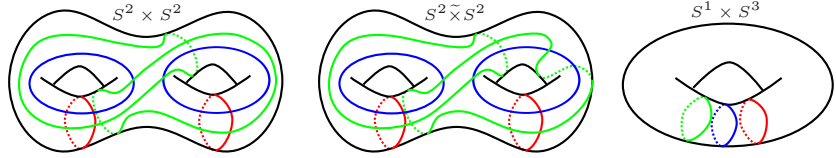}
 \caption{\label{F:VariousGenus2} Various genus $2$ trisection diagrams.}
\end{figure}

\end{itemize}

Now we briefly discuss trisection diagrams more generally. Given a trisection diagram $(F_g,\alpha,\beta,\gamma)$, the $4$--manifold $X(F_g,\alpha,\beta,\gamma)$ is constructed by attaching $4$--dimensional $2$--handles to $F_g \times D^2$ along $\alpha \times \{1\}$, $\beta \times \{e^{2 \pi i/3}\}$ and $\gamma \times \{e^{4\pi i/3}\}$, with framings coming from $F_g \times \{p\}$, and the remainder of $X$ is $3$-- and $4$--handles. Recall that there is a unique way, up to diffeomorphism, to attach the $3$-- and $4$--handles~\cite{LaudenbachPoenaru}.

Since each of $(F_g,\alpha,\beta)$, $(F_g, \beta,\gamma)$ and $(F_g, \gamma, \alpha)$ is a Heegaard diagram for $\sharp^k (S^1 \times S^2)$, each can, after a sequence of handle slides, be made to look like the standard genus $g$ Heegaard diagram of $\sharp^k (S^1 \times S^2)$. However, there is no reason to expect that we can simultaneously arrange for all three pairs of sets of curves to be standard.

Figure~\ref{F:GeneralTrisectionDiagram}  illustrates a general trisection diagram (except that only one $\gamma$ curve is shown) where we have made the $(F_g,\alpha,\beta)$ standard, where $\alpha$ is red and $\beta$ is blue; the reds and blues give the standard genus $g$ Heegaard diagram for $\sharp^k (S^1 \times S^2)$. The important point is that most of the information about the $4$--manifold $X$ is then carried by the $\gamma$ curves (one of which is drawn here in green). These green curves can be drawn anywhere with the proviso that some sequence of handle slides of the greens amongst the greens and the reds amongst reds, followed by a diffeomorphism of $F_g$, can make the reds and greens look like the reds and blues. The same proviso holds for the greens and blues, but a different sequence of handle slides and a different diffeomorphism may be required.
\begin{figure}
\labellist
\tiny\hair 2pt
\pinlabel $k$ [t] at 126 8
\pinlabel $g-k$ [t] at 332 10
\pinlabel $\alpha_1$ [r] at 45 44
\pinlabel $\beta_1$ [l] at 49 42
\pinlabel $\gamma_1$ [b] at 51 96
\pinlabel $\alpha_{k+1}$ [l] at 277 44
\pinlabel $\beta_{k+1}$ [r] at 267 68
\endlabellist
\centering
 \includegraphics[scale=.9]{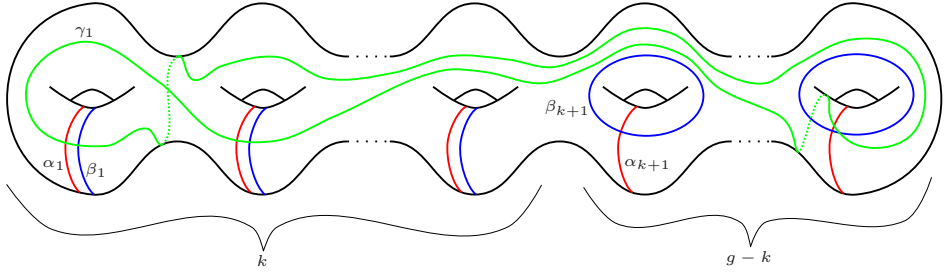}
 \caption{\label{F:GeneralTrisectionDiagram} A general trisection diagram; only one $\gamma$ curve is drawn, although there should be $g$ of them.}
\end{figure}

In fact, if a trisection diagram is drawn so that $\alpha$'s and $\beta$'s are standard as in Figure~\ref{F:GeneralTrisectionDiagram}, then a framed link diagram for $X(F_g,\alpha,\beta,\gamma)$ is obtained by erasing the last $(g-k)$ $\alpha$'s and $\beta$'s (which appear as meridian--longitude pairs) and then replacing each of the first $k$ parallel pairs of $\alpha$'s and $\beta$'s by a parallel dotted circle ($1$--handle) pushed slightly out of $F_g$. The $\gamma$'s remain as the attaching maps for $2$--handles, and their framings come from the surface $F_g$.

\textbf{An extended example: $3$--manifold bundles over $S^1$.} (Thanks to Stefano Vidussi for asking interesting questions that led to this example.) Suppose $X^4$ fibers over $S^1$, $M \into X \to S^1$, with fiber a closed, connected, oriented $3$--manifold $M^3$, and monodromy $\mu: M \to M$.

A trisection of $X$ is not immediately obvious, just as a bisection (Heegaard splitting) is not immediate when a $3$--manifold fibers over a circle: $F_g \into M \to S^1$.

In the latter case, one takes two fibers over distinct points of $S^1$, separating $M$ into two copies of $I \times F$.  Choose a Morse function on $F$ with one critical point of index $2$ and thus one $2$--handle $H$.  Remove $I \times H$ from one $I \times F$ and add it to the other copy of $I \times F$.  This turns the first copy into a handle body with $2g$ $1$--handles, and adds a $1$--handle to the second copy.  Again let $H$ be the $2$--handle of the second copy (disjoint from the first $H$), and remove $I\times H$ from the second copy of $I \times F$ and add it to the first copy.  Now both copies are handle bodies with $2g+1$ $1$--handles and we have the desired Heegaard spitting.

In the $4$--dimensional case, $X^4 = S^1 \times_{\mu} M$, pick a Morse function $\tau_0 : M \to [0,3]$ with only one critical point $\hat{x}$ of index $3$ and only one $\bar{x}$ of index $0$ ($\tau_0$ could give a minimal genus Heegaard splitting if desired).

Then $\tau_0 \mu$ is another Morse function on $M$ with the same kind of critical points, and $\mu$ can be isotoped so as to fix the maximum $\hat{x}$ and the minimum $\bar{x}$.  Then there is a homotopy $\tau_t: M \to [0,3], t \in [0,1]$, such that

\begin{enumerate}

\item $\tau_1 = \tau_0 \mu$, 

\item $\tau_t = \tau_0 = \tau_0 \mu$ on $\hat{x}$ and $\bar{x}$ and there are no other definite critical points of $\tau_t$,

\item $\tau_t$ is a Morse function for all but a finite number of values of $t$ at which $\tau_t$ has a birth or a death of a cancelling pair of indefinite critical points.

\end{enumerate}

Since $S^1 = [0,1]/0 \sim 1$, property (1) allows us to define
$$\tau: X^4 = ([0,1] \times M)/(1,x) \sim (0,\mu (x)) \to S^1 \times [0,3]$$
by setting $\tau (t,x) = (t,\tau_t (x))$.  To check, note that $\tau(1,x) = (1,\tau_1(x)) = (0, \tau_0 (\mu (x))) = \tau (0, \mu (x))$.

Thus we have a smoothly varying family of Morse functions on the fibers of $X$, except for the births and deaths.  There are an equal number of births and deaths because $\tau_0$ and $\tau_0 \mu$ have the same number of critical points.  Then we can make all the births happen earlier at $t=0$ and the deaths later at $t=1$, and furthermore by an isotopy of $\mu$, the births and deaths can be paired off and happen at the same points of $M$.  In that case the pairs can be merged and then $\tau$ is a family of Morse functions of the fibers of $X$ with only one fixed maximum and minimum and $g$ critical points of indices $1$ and $2$.  Furthermore, it is straightforward to arrange that all critical points of index $1$ (resp $2$) take values in a small neighborhood of $1$ (resp. $2$) for each $t \in S^1$.

Now draw a hexagonal-like grid on $[0,1]\times [0,3]$ as in Figure~\ref{F:FiberOverS1} and label the boxes with $X_i, i = 1,2,3$.  Recall that the left and right ends are identified so as to have $S^1 \times [0,3]$.
\begin{figure}
\labellist
\tiny\hair 2pt
\pinlabel $X_1$ at 19 31
\pinlabel $X_2$ at 64 31
\pinlabel $X_3$ at 123 31
\pinlabel $X_1$ at 168 31
\pinlabel $X_3$ at 34 65
\pinlabel $X_1$ at 94 65
\pinlabel $X_2$ at 153 65
\pinlabel $(0,0)$ [tr] at 6 4
\pinlabel $(0,1)$ [r] at 4 32
\pinlabel $(0,2)$ [r] at 4 64
\pinlabel $(0,3)$ [r] at 4 92
\pinlabel $(1,0)$ [t] at 183 3
\pinlabel $(1,3)$ [l] at 183 92
\pinlabel $F_g$ [bl] at 64 55
\pinlabel {$H_g = \natural^g S^1 \times B^2$} [t] at 82 2
\pinlabel {\rotatebox{34}{$I \times F_g$}} [l] at 100 42

\endlabellist
\centering
 \includegraphics{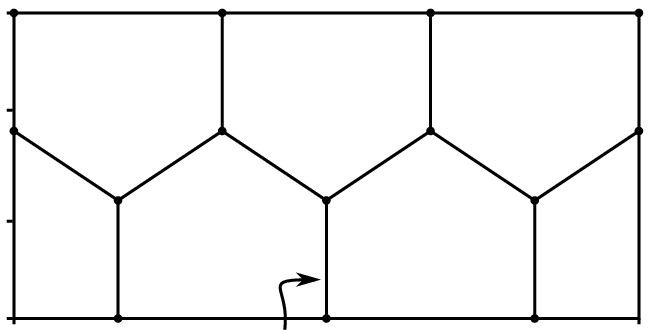}
 \caption{\label{F:FiberOverS1} Fibering over $S^1$.}
\end{figure}

The trisection of $X$ into $X_1 \cup X_2 \cup X_3$ is to be made by tube-connect summing the pre images under $\tau$ of the $X_i$'s in Figure~\ref{F:FiberOverS1}.  Over each vertical line segment in Figure~\ref{F:FiberOverS1} is $H_g$ which is defined to be a $3$--dimensional handle body with $g$ $1$--handles, so over the interior vertices lie surfaces $F_g$.  Over the diagonally sloped line segments lie $3$--manifolds $I \times F_g$.

Let $H$ be the $2$--handle in $F_g$ and define a $4$--dimensional $1$--handle to be a thickening of $I \times H$ into the bounding $X_i$'s on either side of $I \times Fg$.  Add such a $1$--handle to connect each $X_i$ to another $X_i$ across a sloping line segment, for $i = 1,2,3$.  Doing this twice for $X_1$, once along a SW-NE sloping line and once along a NW-SE sloping one as in Figure~\ref{F:FiberOverS1Tubing}, we see that $X_1$ has become connected and is a $4$--dimensional handlebody with $2g+1$ $1$--handles.  Similarly with $X_2$ and $X_3$.
\begin{figure}
\labellist
\tiny\hair 2pt
\pinlabel {$\natural^g S^1 \times B^3$} [r] at 2 24
\pinlabel $X_2$ at 64 31
\pinlabel $X_3$ at 123 31
\pinlabel $X_3$ at 34 65
\pinlabel {$\natural^g S^1 \times B^3$} at 94 65
\pinlabel $X_2$ at 153 65
\pinlabel {$1$--h} [tl] at 48 48
\pinlabel {$1$--h} [bl] at 138 48

\endlabellist
\centering
 \includegraphics{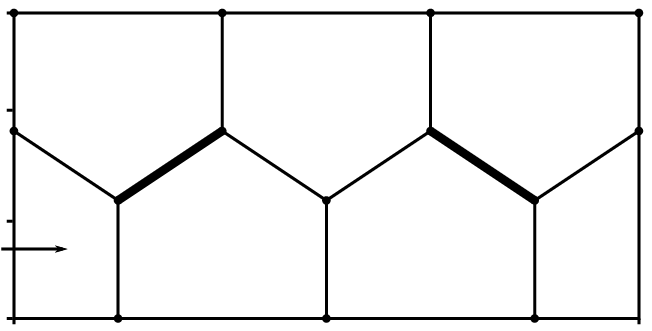}
 \caption{\label{F:FiberOverS1Tubing} Connect the regions with $1$--handles; here the $1$--handles connecting the $X_1$'s are highlighted.}
\end{figure}

Next we calculate $X_1 \cap X_2$. Its various parts are shown in Figure~\ref{F:FiberOverS1Pairs}.  Note that the sloping edges with labels $H_{2g}$ arise from $I\times F_g$ by having removed the $I \times H$.  Thus we have $H_g \cup H_{2g} \cup H_g \cup H_{2g} \cup 4$ $1$--handles, and three of the $1$--handles cancel $0$--handles leaving $H_{6g+1} = X_1 \cap X_2 = X_2 \cap X_3 = X_3 \cap X_1$.  Then the central fiber $F_{g'}$ of the trisection has genus $g' = 6g+1$ and gives a Heegaard splitting of $\partial X_i = \#_{2g+1} S^1 \times S^2$.  Note that $k = 2g+1$ and we can check that $\chi(X) = 0 = 2 + g' -3k$.
\begin{figure}
\labellist
\tiny\hair 2pt
\pinlabel $X_1$ at 19 28
\pinlabel $X_2$ at 64 28
\pinlabel $X_3$ at 123 28
\pinlabel $X_1$ at 168 28
\pinlabel $X_3$ at 34 68
\pinlabel $X_1$ at 94 68
\pinlabel $X_2$ at 153 68
\pinlabel {$1$--h} [bl] at 20 48
\pinlabel {$1$--h} [tl] at 48 48
\pinlabel {$1$--h} [tl] at 110 48
\pinlabel {$1$--h} [bl] at 138 48
\pinlabel $H_g$ [l] at 36 20
\pinlabel $H_g$ [l] at 125 75
\pinlabel {$H_{2g}$} [bl] at 79 48
\pinlabel {$H_{2g}$} [tl] at 168 48

\endlabellist
\centering
 \includegraphics{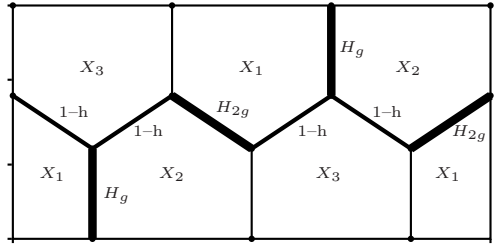}
 \caption{\label{F:FiberOverS1Pairs} Understanding the pairwise intersections when fibering over $S^1$.}
\end{figure}

\textbf{Surface bundles over $S^2$.} Now suppose that $X^4$ fibers over $S^2$ with fiber $F$ a closed surface of genus $g_F$. We construct a trisection in a similar fashion to the preceding example. 

Let $\pi : X \to S^2$ be the fibration. Identify $S^2$ with a cube and trisect $S^2$ as $S^2 = A_1 \cup A_2 \cup A_3$, where each $A_i$ is the union of two opposite (closed) faces of the cube. Choose disjoint sections $\sigma_1$, $\sigma_2$ and $\sigma_3$ over $A_1$, $A_2$ and $A_3$, respectively, and let $N_i$ be a closed tubular neighborhood of $\sigma_i$, for $i=1,2,3$, with the $N_i$'s also disjoint. The trisection of $X$ is $X=X_1 \cup X_2 \cup X_3$ where
\[ X_i = (\pi^{-1}(A_i) \setminus \mathring{N}_i) \cup N_{i+1}, \]
with indices taken mod $3$.

We now verify that this is indeed a trisection, and compute $g$ and $k$ along the way. First, $\pi^{-1}(A_i)$ is two copies of $D^2 \times F$. Next, removing $\mathring{N}_i$ leaves us with two copies of $D^2 \times F'$, where $F'$ has genus $g_F$ and one boundary component. Thus $\pi^{-1}(A_i) \setminus \mathring{N}_i$ has two $0$--handles and $4 g_F$ $1$--handles. Finally, $N_{i+1}$ is two $1$--handles connecting the two components of $\pi^{-1}(A_i) \setminus \mathring{N}_i$. Thus one of the $0$--handles is cancelled by one of these two $1$--handles, and we are left with one $0$--handle and $k=4g_F+1$ $1$--handles.

Now we consider the pairwise intersections. The $3$--dimensional intersection $X_1 \cap X_2$ is the union of four pieces:
\begin{itemize}
 \item $(\pi^{-1}(A_1) \setminus \mathring{N}_1) \cap (\pi^{-1}(A_2) \setminus \mathring{N}_2)$: Since $A_1$ and $A_2$ intersect along four edges of the cube, this is four copies of $[0,1] \times F''$, where $F''$ has genus $g_F$ and {\em two} boundary components. In other words, this $3$--manifold is built from four $0$--handles and $4(2g_F+1)=8g_F+4$ $1$--handles.
 \item $(\pi^{-1}(A_1) \setminus \mathring{N}_1) \cap N_3$: This sits over the four edges making up $A_1 \cap A_3$, and thus contributes four $1$--handles, two connecting two of the components above, and two connecting the other two. Cancelling two of the $0$--handles from the preceding step with two of these $1$--handles, we are left with two $0$--handles and $8g_F+6$ $1$--handles.
 \item $N_2 \cap (\pi^{-1}(A_2) \setminus \mathring{N}_2)$: This is just $\partial N_2$, which is two copies of $D^2 \times S^1$, joining up the four copies of $[0,1] \times F''$, from the first step above, in pairs, with the $S^1$ factor in $D^2 \times S^1$ lining up with one of the boundary components of the $F''$ factor of $[0,1] \times F''$. Thus we get two new $1$--handles and two new $2$--handles. One of the $1$--handles cancels a $0$--handle, and both $2$--handles cancel $1$--handles. This leaves us with one $0$--handle and $8g_F+6+1-2=8g_F+5$ $1$--handles.
 \item $N_2 \cap N_3$: This is empty.
\end{itemize}
Thus $X_1 \cap X_2$ is a $3$--dimensional handlebody with genus $g=8g_F+5$, and the same holds for $X_2 \cap X_3$ and $X_3 \cap X_1$.

The triple intersection is necessarily the boundary of each pairwise intersection, so we see that we have a trisection with $k=4g_F+1$ and $g=8g_F+5$. This gives $\chi = 2+g-3k = 4-4g_F$, which is what we expect for a genus $g_F$ bundle over $S^2$.

When this technique is applied to $S^2 \times S^2$ we get the genus $5$ diagram in Figure~\ref{F:S2XS2Cube}. With some work this can be shown to be handle slide and diffeomorphism equivalent to a single stabilization of the genus $2$ diagram of $S^2 \times S^2$ in Figure~\ref{F:VariousGenus2}.
\begin{figure}
\labellist
\tiny\hair 2pt
\endlabellist
\centering
 \includegraphics[width=2in]{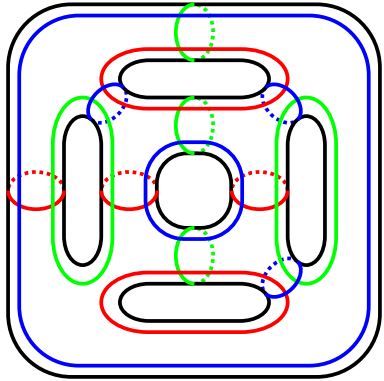}
 \caption{\label{F:S2XS2Cube} A genus $5$ trisection diagram for $S^2 \times S^2$ obtained by seeing $S^2 \times S^2$ as an $S^2$ bundle over a cube. The surface shown here is naturally the boundary of a tubular neighborhood of the $1$--skeleton of a cube.}
\end{figure}

\textbf{Gluing maps:} A $4$--manifold $X$ with a trisection $(X_1,X_2,X_3)$ is determined {\em up to
diffeomorphism} by the data of $k$, $g$ and three gluing maps between the sectors; see Figure~\ref{F:Gluing}. Here we discuss this gluing data carefully and show how to reduce the data to two elements of the mapping class group of a closed genus $g$ surface satisfying certain constraints. 
\begin{figure}
\labellist
\tiny\hair 2pt
\pinlabel $X_1$ at 22 71
\pinlabel $X_2$ at 53 20
\pinlabel $X_3$ at 80 73
\pinlabel $\phi_1$ [tr] at 43 50
\pinlabel $\phi_2$ [tl] at 60 49
\pinlabel {$\phi_1^{-1}\phi_2^{-1}$} [b] at 52 65
\pinlabel $\psi_1$ [tr] at 25 40
\pinlabel $\psi_2$ [tl] at 78 38
\pinlabel $\psi_3$ [b] at 52 86
\pinlabel $Y_{k,g}^+$ [r] at 41 84
\pinlabel $Y_{k,g}^-$ [br] at 23 48
\pinlabel $Y_{k,g}^+$ [tl] at 32 31
\pinlabel $Y_{k,g}^-$ [tr] at 74 31
\pinlabel $Y_{k,g}^+$ [bl] at 82 48
\pinlabel $Y_{k,g}^-$ [l] at 63 85
\pinlabel $F_g$ at 51 55
\endlabellist
\centering
 \includegraphics[scale=2]{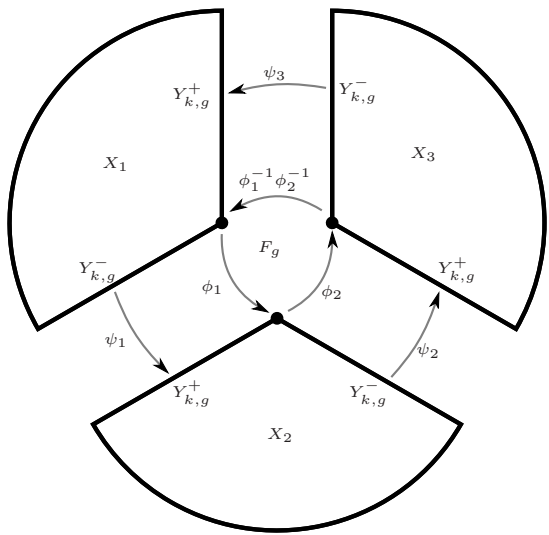}
 \caption{\label{F:Gluing} Gluing maps.}
\end{figure}

Let $X_1$, $X_2$ and $X_3$ be copies of $Z_k = \natural^k (S^1 \times B^3)$. Let $Y_k = \partial Z_k = Y_{k,g}^+ \cup Y_{k,g}^-$ be the
standard genus $g$ Heegaard splitting of $Y_k = \sharp^k (S^1 \times S^2)$ with $H_{k,g} = Y_{k,g}^+ \cap Y_{k,g}^-$ the Heegaard surface, with a fixed identification $H_{k,g} \cong F_g$.
We can then construct a $4$--manifold with three diffeomorphisms $\psi_i \co Y_{k,g}^- \to - Y_{k,g}^+$, for $i=1,2,3$, such that $\psi_i$ glues $X_{i}$ to $X_{i+1}$ (indices taken mod. $3$) by gluing the copy of $Y_{k,g}^-$ in $\partial X_{i}$ to the copy of $Y_{k,g}^+$ in $\partial X_{i+1}$. Let $\phi_i = \psi_i |_{F_g} \co F_g \to F_g$ and note that we need $\phi_3 \circ \phi_2 \circ \phi_1$ to be isotopic to the identity in order for the resulting manifold to close at the central fiber $F_g$. Furthermore, since an automorphism of a $3$--dimensional handlebody is completely determined up to isotopy by its restriction to the boundary surface, this entire construction is actually determined by the two (isotopy classes of) maps $\phi_1,\phi_2 \co F_g \to F_g$, with $\phi_3 = \phi_1^{-1} \circ \phi_2^{-1}$.

However, this characterization is slightly misleading because an arbitrary pair $\phi_1, \phi_2$ of mapping classes of $F_g$ does not necessarily produce a trisected $4$--manifold: We need that each of $\phi_1$, $\phi_2$ and $\phi_1^{-1} \circ \phi_2^{-1}$ extends to a diffeomorphism $\psi_i \co Y_{k,g}^- \to - Y_{k,g}^+$, a slightly messy condition that is not entirely trivial to check.

\textbf{Gluing maps from model manifolds:} In fact we can reduce the gluing map data to a single gluing map if we construct trisected $4$--manifolds by cutting open and regluing fixed model trisected manifolds. For each $0 \leq k \leq g$ let $X^{k,g} = (\sharp^k S^1 \times S^3) \sharp (\sharp^{g-k} \C P^2)$. Note that $X^{k,g}$ has a standard $(k,g)$--trisection $X^{k,g}=(X^{k,g}_1,X^{k,g}_2,X^{k,g}_3)$, because $S^1 \times S^3$ has a standard $(1,1)$--trisection and $\C P^2$ has a standard $(0,1)$--trisection. Also, for each such $(k,g)$, fix an identification of $X^{k,g}_1 \cap X^{k,g}_2$ with the standard genus $g$ handlebody $H_g = \natural^g S^1 \times B^2$. Then any other $4$--manifold $X$ with a $(k,g)$--trisection is obtained from $X^{k,g}$ by cutting $X^{k,g}_1$, $X^{k,g}_2$ and $X^{k,g}_3$ apart, regluing $X^{k,g}_1$ to $X^{k,g}_2$ by some automorphism $\phi$ of $X^{k,g}_1 \cap X^{k,g}_2 = H_g$, and then observing that gluing in $X^{k,g}_3$ amounts to attaching a collection of $3$--handles and a $4$--handle, so that no other gluing data needs to be specified. Again, not any automorphism $\phi: H_g \to H_g$ will work, but now one needs to verify that $\partial(X^{k,g}_1 \cup_\phi X^{k,g}_2)$ is diffeomorphic to $\sharp^k (S^1 \times S^2)$ in order to verify that $\phi$ actually produces a closed trisected $4$--manifold.

\textbf{Lagrangians, Maslov index, signature and intersection triples:}
Given a genus $g$ trisection diagram $(F_g,\alpha,\beta,\gamma)$, one can write down a triple $(Q_{\alpha \beta}, Q_{\beta \gamma}, Q_{\gamma \alpha})$ of $g \times g$ integer matrices, giving the intersection pairing between curves. Our uniqueness theorem tells us that this {\em intersection triple} is uniquely determined by the diffeomorphism type of $X(F_g,\alpha,\beta,\gamma)$ up to elementary row-column operations and stabilization. Here, the row-column operations are precisely those corresponding to handle slides. Thus, for example, sliding $\alpha_1$ over $\alpha_2$ corresponds to adding row $2$ to row $1$ in $Q_{\alpha \beta}$ while {\em simultaneously} adding column $2$ to column $1$ in $Q_{\gamma \alpha}$. Stabilization replaces $(Q_{\alpha \beta}, Q_{\beta \gamma}, Q_{\gamma \alpha})$ with the following triple:
\[ \left( \left[ \begin{array}{c|c}
                  Q_{\alpha \beta} & 0 \\
                  \hline
                  0 & \begin{array}{ccc} 1 & 0 & 0 \\ 0 & 1 & 0 \\ 0 & 0 & 0 \end{array}
                 \end{array} \right],
          \left[ \begin{array}{c|c}
                  Q_{\beta \gamma} & 0 \\
                  \hline
                  0 & \begin{array}{ccc} 1 & 0 & 0 \\ 0 & 0 & 0 \\ 0 & 0 & 1 \end{array}
                 \end{array} \right],
          \left[ \begin{array}{c|c}
                  Q_{\gamma \alpha} & 0 \\
                  \hline
                  0 & \begin{array}{ccc} 0 & 0 & 0 \\ 0 & 1 & 0 \\ 0 & 0 & 1 \end{array}
                 \end{array} \right]
 \right) \]
The fact that each pair of collections of curves gives a Heegaard diagram for $\sharp^k S^1 \times S^2$ tells us that each of the three matrices is, independently, row-column equivalent to $\left[ \begin{array}{cc} 0_k & 0 \\ 0 & I_{g-k} \end{array} \right]$. We thus have an invariant of $4$--manifolds taking values in this set of triples, subject to this $\sharp^k S^1 \times S^2$ condition, modulo an interesting equivalence relation. Of course, this invariant may contain nothing more than homological information, for example, but even if that were true it would be interesting to understand exactly how this works.

Alternatively, one can define three Lagrangian subspaces $(L_\alpha, L_\beta, L_\gamma)$ in the symplectic vector space $V = H_1(F_g ; \R)$; i.e. $L_\alpha$ is the kernel of the map $H_1(F_g; \R) \to H_1(H_\alpha; \R)$ where $H_\alpha$ is the handlebody determined by the $\alpha$ curves, and so on. One immediately recovers the three intersection matrices above as the symplectic form on $V$ restricted to each pair of Lagrangians, relative to chosen bases on the Lagrangians. Thus our uniqueness theorem also gives us a $4$--manifold invariant taking values in the set of quadruples $(V,L_\alpha,L_\beta,L_\gamma)$, where $V$ is a symplectic vector space and the $L$'s are linear Lagrangian subspaces, subject again to the $\sharp^k S^1 \times S^2$ condition, modulo an equivalence relation. This equivalence relation is linear symplectomorphism and stabilization, which in this case means taking the direct sum with $(\R^6,\langle x_1,x_2,y_3 \rangle, \langle x_1,y_2,x_3 \rangle, \langle y_1,x_2,x_3 \rangle)$.

This Lagrangian set up has in fact been studied in the more general context of Wall's~\cite{wall} nonadditivity of the signature. A direct application of the interpretation in~\cite{clm} of Wall's nonadditivity result shows that the signature of a closed $4$--manifold with a trisection is precisely the Maslov index of this associated triple of Lagrangians.

However, one expects more information to be encoded in these Langrangians triples than just the signature. In particular, the Maslov index ignores the integer lattice structure of $H_1(F_g,\Z) \subset H_1(F_g,\R)$. Quotienting out by this lattice gives us a triple of Lagrangian $g$--tori in a symplectic $2g$--torus, and one again gets a $4$--manifold invariant taking values in these triples mod symplectomorphism and stabilization. It seems that a further study of this set up could be fruitful.

\textbf{Curve complex perspective:} To record much more data than simply the homology classes of curves which bound disks in the three handlebodies, we can consider, for each handlebody $H_{12}$, $H_{23}$ and $H_{31}$, the subsets $U_{12}$, $U_{23}$ and $U_{31}$, respectively of the curve complex for $F_g$ given by those essential simple closed curves which bound disks in the respective handlebody. Because each pair of handlebodies gives $\sharp^k (S^1 \times S^2)$, we know that the three intersections $U_{12} \cap U_{23}$, $U_{23} \cap U_{31}$ and $U_{31} \cap U_{12}$ are nonempty. This perspective raises many interesting questions, such as: What is the minimal area of a triangle with vertices in the three intersections? If $U_{12} \cap U_{23} \cap U_{31}$ is nonempty, what does that tell us about $X$? If the gluing map coming from the model manifold construction described above is, for example, pseudo-Anosov, does this tell us that the three subcomplexes are ``far apart'' in any sense?

\section{Existence via Morse $2$--functions}

The proof presented in this section is an
application of tools developed in~\cite{gk1}, using Morse $2$--functions. In the following section we will rewrite the
proof entirely in terms of ordinary Morse functions
and handle decompositions, but the trisection is so natural from the
point of view of Morse $2$--functions that we feel this proof is worth presenting. However, to give the basic idea for those most comfortable
with the language of handle decompositions, our construction ends up
putting the $0$-- and $1$--handles of $X$ into $X_1$, the $3$-- and
$4$--handles into $X_3$, and the $2$--handles together with some
``connective tissue'' into $X_2$.

A Morse $2$--function is a smooth, stable map $G \co X^n \to
\Sigma^2$; in this paper we will always map to $\R^2$. (Stable implies
generic when mapping to dimension two.) Just like Morse functions,
Morse $2$--functions can be characterized by local models, and we now
give these local models only in the case of $n=4$, i.e. we are
considering an $\R^2$--valued Morse $2$--function $G$ on a
$4$--manifold $X$:
\begin{enumerate}
 \item Each regular value $q \in \R^2$ has a coordinate neighborhood
   over which $G$ looks like $F^2 \times B^2 \to B^2$ for some closed
   fiber surface $F$.
 \item The set of critical points of $G$ is a smooth $1$--dimensional
   submanifold $\Crit_G \subset X$ such that $G \co \Crit_G \to \R^2$
   is an immersion with isolated semicubical cusps and crossings. The
   non-cusp points of $\Crit_G$ are called {\em fold points}, and arcs
   of such points are called {\em folds}.
 \item Each point $q \in G(\Crit_G)$ which is not a cusp or crossing
   has a neighborhood $U = I \times I$ with coordinates $(t,y)$, with
   $G^{-1}(U)$ diffeomorphic to $I \times M^3$ for a $3$--dimensional
   cobordism $M$, so that $G(t,p) = (t,g(p))$, where $g: M \to I$ is a
   Morse function on $M$ with one critical point. The index of this
   critical point is then called the index of the fold, although this
   is only well-defined up to $i \mapsto 3-i$. When the image of the
   fold is co-oriented, the index is well-defined by insisting that
   the $y$--coordinate on $I \times I$ increases in the direction of
   this co-orientation. 
 \item Each cusp point $q \in G(\Crit_G)$ has a neighborhood $U = I
   \times I$ with coordinates $(t,y)$, with $G^{-1}(u) = I \times
   M^3$, so that $G(t,p) = (t,g_t(p))$, where $g_t$ is a
   $1$--parameter family of Morse functions on $M$ with no critical
   points for $t=0$ and a birth of a cancelling pair of critical
   points at $t=1/2$. In our examples, these two critical points will
   always be of index $1$ and $2$. 
 \item Each crossing point $q \in G(\Crit_G)$ has a neighborhood $U =
   I \times I$ with coordinates $(t,y)$, with $G^{-1}(u) = I \times
   M^3$, so that $G(t,p) = (t,g_t(p))$, where $g_t$ is a
   $1$--parameter family of Morse functions on $M$ with two critical
   points for all $t$, such that the critical values cross at
   $t=1/2$. In our examples, these two critical points will never be
   of index $0$ or $3$. 
\end{enumerate}

The basic example of a Morse $2$--function is $(t,p) \mapsto
(t,g_t(p))$ for an arbitrary generic homotopy $g_t$ between two given
Morse functions $g_0, g_1 \co M^3 \to [0,1]$, and the message of the
above local models is that Morse $2$--functions look locally like
homotopies between Morse functions, but globally we may not have a
preferred ``time'' direction. When $G$ is of the form $(t,p) \mapsto
(t,g_t(p))$, we call $G(\Crit_G)$ a {\em Cerf
  graphic}~\cite{cerf}. Conversely, given a Morse $2$--function $G \co
X^4 \to \R^2$ and a rectangle $I \times I \subset \R^2$ in which
$G(\Crit_G)$ has no vertical tangencies, we can find coordinates in
which $G$ is of this form $(t,p) \mapsto (t,g_t(p))$, and so again we
will say that $G(\Crit_G)$ is a Cerf graphic in this rectangle.

There is one move on Morse $2$--functions (i.e. local model for a generic homotopy between Morse $2$--functions) that is central to this paper, which we call the ``introduction of an eye''. In a local chart in which a given Morse $2$--function $G$ on a $4$--manifold has no critical points, we can assume $G$ has the form $(t,x,y,z) \mapsto (t,x)$ or, equivalently, $(t,x,y,z) \mapsto (t,x^3+(t^2+1)x-y^2+z^2)$ with $t \in [-2,2]$. Introducing a parameter $s \in [-1,1]$ we get a homotopy $(t,x,y,z) \mapsto (t,x^3+(t^2-s)x-y^2+z^2)$, with $s=-1$ corresponding to the given map and $s=1$ the end result of ``introducing an eye''. Figure~\ref{F:Eye} shows the image of the critical locus at $s=1$, justifying the terminology. Note that this is a Cerf graphic in which, as $t$ increases from $-2$ to $2$, we see a Morse function on $x,y,z$ space which starts with no critical points, develops a cancelling pair of index $1$ and $2$ critical points, and then the cancelling pair disappears again so that at 
$t=2$ there are again no critical points.
\begin{figure}
\labellist
\tiny\hair 2pt
\pinlabel $z$ [b] at 64 75
\pinlabel $t$ [l] at 128 38
\endlabellist
\centering
 \includegraphics{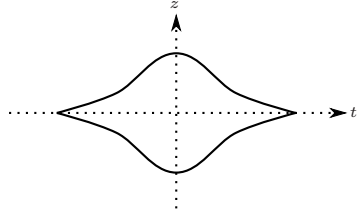}
 \caption{\label{F:Eye} The ``eye'', a Cerf graphic in which a pair of cancelling critical points is born and then dies.}
\end{figure}
Note also that the introduction of an eye takes place in a ball and is localized to a disk in the fiber cross a disk in the base; thus, as long as fibers are connected, we need only specify a disk in the base without critical points and then there is a unique, up to isotopy, way to introduce an eye in that disk.

\begin{proof}[Proof of Existence, Theorem~\ref{T:Existence}]
 Throughout we will use coordinates $(t,z)$ on $\R^2$, with $t$
 horizontal and $z$ vertical. Here is an outline of the proof: 
\begin{enumerate}
 \item First we will show that there is a Morse $2$--function $G_1 \co
   X \to \mathbb{R}^2$ such that the image of the fold locus is as in
   Figure~\ref{F:Folds1}. In this and the following figures, three
   dots between two curves indicate that there are some number of
   parallel copies of the two curves in between. Fold indices are
   indicated with labelled transverse arrows. Boxes with folds coming
   in from the left and out at the right represent arbitrary Cerf
   graphics, with the left-right axis being time. Note that a Cerf
   graphic may contain left-cusps, right-cusps and crossings, but may
   not contain any vertical tangencies on the image of the fold
   locus. 
\begin{figure}
\labellist
\tiny\hair 2pt
\pinlabel $t$ [l] at 353 17
\pinlabel $z$ [b] at 16 209
\pinlabel $0$ [t] at 40 16
\pinlabel $1$ [t] at 56 16
\pinlabel $1$ [t] at 72 16
\pinlabel $2$ [t] at 116 16
\pinlabel $2$ [t] at 140 16
\pinlabel $3$ [t] at 284 16
\pinlabel $3$ [t] at 304 16
\pinlabel $4$ [t] at 320 16
\pinlabel $0$ [tl] at 180 40
\pinlabel $3$ [tl] at 180 200
\pinlabel $1$ [tl] at 188 56
\pinlabel $2$ [tl] at 188 136
\pinlabel $1$ [tl] at 256 60
\pinlabel $2$ [tl] at 256 160
\endlabellist
\centering
 \includegraphics{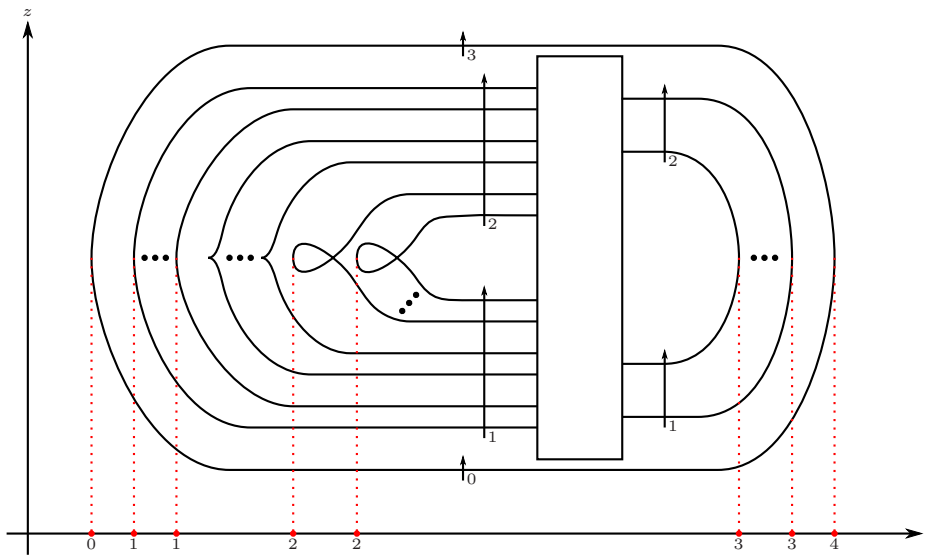}
 \caption{\label{F:Folds1} The image of the fold locus for $G_1$.}
\end{figure}

 \item In Figure~\ref{F:Folds1}, the vertical tangencies of the folds
   are highlighted in red; these become critical points of the
   projection $t \circ G_1 \co X \to \R$. These critical values in
   $\R$ are also indicated at the bottom of the diagram along the
   $t$--axis, with their indices. 

 \item After constructing $G_1$, we will show to homotope $G_1$ to
   $G_2$ such that the image of the fold locus for $G_2$ is as in
   Figure~\ref{F:Folds2}. Here the two Cerf graphics have no cusps. We
   have achieved two goals here: (1) Splitting the Cerf graphic into
   two, each involving only critical points of the same index and no
   cusps. (2) Replacing each kink that corresponds to an index $2$
   critical point of $t \circ G_1$ with a pair of cusps. 
\begin{figure}
\labellist
\tiny\hair 2pt
\pinlabel $0$ [tl] at 180 40
\pinlabel $1$ [tl] at 188 56
\pinlabel $2$ [tl] at 188 136
\pinlabel $3$ [tl] at 180 200
\pinlabel $1$ [tl] at 236 60
\pinlabel $2$ [tl] at 236 140
\endlabellist
\centering
 \includegraphics{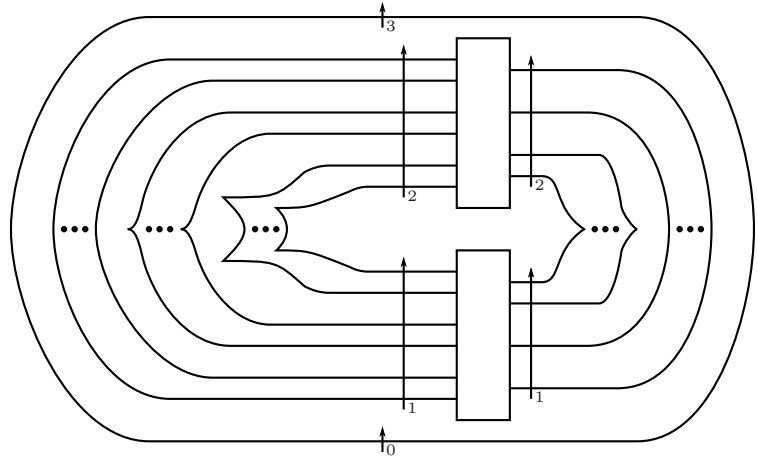}
 \caption{\label{F:Folds2} The image of the fold locus for $G_2$.}
\end{figure}

 \item Figure~\ref{F:Folds3} is simply a redrawing of
   Figure~\ref{F:Folds2} that highlights a natural trisection of
   $\R^2$ into three sectors $\R^2_1$, $\R^2_2$ and $\R^2_3$. Note
   that the critical locus over each sector consists of $g$
   components, where $g$ is the genus of the central fiber. Also, each
   such component has at most one cusp. We no longer indicate the
   indices of the folds; the outermost fold is index $0$ pointing
   inwards, and all other folds are index $1$ pointing in. 
\begin{figure}
\labellist
\tiny\hair 2pt
\pinlabel $\R^2_1$ at 24 206
\pinlabel $\R^2_2$ [tr] at 100 5
\pinlabel $\R^2_3$ [bl] at 234 164
\endlabellist
\centering
 \includegraphics[width=3in]{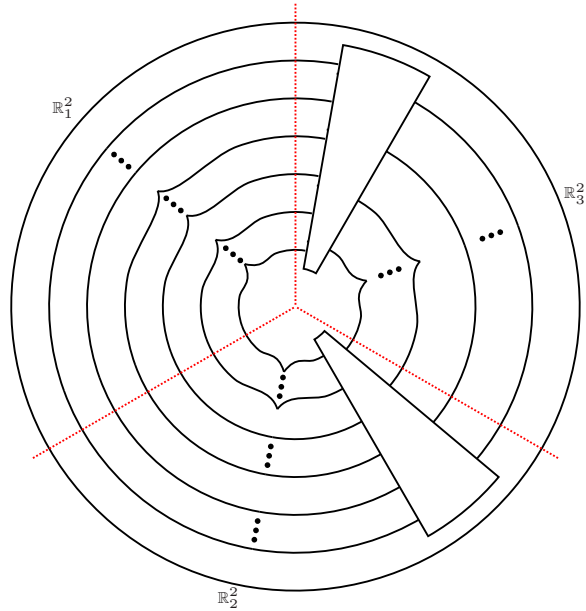}
 \caption{\label{F:Folds3} A more symmetric drawing of the image of the fold locus for $G_2$. We no longer indicate the indices of the folds; the outermost fold is index $0$ going inwards, the others are index $1$ going inwards.}
\end{figure}

 \item The form of the folds in Figure~\ref{F:Folds3} is a special
   case of the form shown in Figure~\ref{F:Folds4}, where now we are
   not paying attention to which folds in a given sector, with or
   without cusps, connect to which folds in the next sector, with or
   without cusps, and we allow for arbitrary Cerf graphics (without
   cusps) between the sectors. 
\begin{figure}
\centering
 \includegraphics[width=3in]{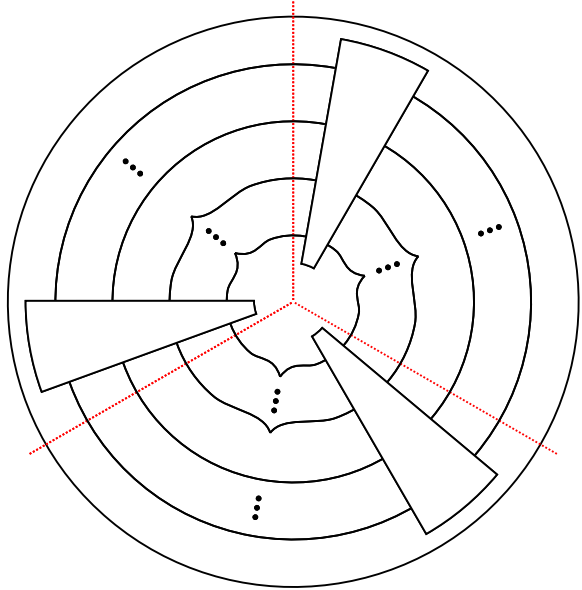}
 \caption{\label{F:Folds4} A slightly more general form for the image
   of the fold locus, which fits $G_2$.} 
\end{figure}
 
 \item Now we have $G_2$ such that the image of the fold locus is as
   in Figure~\ref{F:Folds4}. At this point we could take $X_i =
   G_2^{-1}(\R^2_i)$ and we would have each $X_i$ diffeomorphic to
   $\natural^{k_i} S^1 \times B^3$ for different $k_i$'s. There is one
   last step to arrange that the $k_i$'s are equal: In fact, $k_i$ is
   equal to the number of folds in sector $X_i$ {\em without
     cusps}. We will show how to add a fold without a cusp to any one
   sector while adding a fold with a cusp to each of the other two
   sectors. This allows us to construct a homotopy from $G_2$ to
   $G_3$, such that $G_3$ has the image of its fold locus of the same
   form as $G_2$ (i.e. as in Figure~\ref{F:Folds4}), with the same
   number of folds without cusps in each sector,
   i.e. $k_1=k_2=k_3=k$. 

 \item Finally we will justify the claim that each $X_i =
   G_3^{-1}(\R^2_i)$ is diffeomorphic to $\natural^k S^1 \times B^3$
   with overlap maps as advertised. 
\end{enumerate}

We now fill in the details.

Begin with a handle decomposition of $X$ with one $0$--handle, $i_1$
$1$--handles, $i_2$ $2$--handles, $i_3$ $3$--handles and one
$4$--handle. The union of the $0$-- and $1$--handles, $X_1$ is
diffeomorphic to $I \times (\natural^{i_1} S^1 \times B^2)$. Map this
to $I \times I$ by $(t,p) \mapsto (t,g(p))$ where $g \co
\natural^{i_1} S^1 \times B^2 \to I$ is the standard Morse function
with one index $0$ critical point and $i_1$ index $1$ critical
points. Post-compose this map with a diffeomorphism from $I \times I$
to a half-disk and we have constructed $G_1$ on the union of the $0$--
and $1$--handles so that the image of the fold locus is as in the
right half of Figure~\ref{F:Folds1A}. 
\begin{figure}
\labellist
\tiny\hair 2pt
\pinlabel $0$ [tl] at 37 33
\pinlabel $1$ [tl] at 97 52
\pinlabel $0$ [t] at 176 65
\pinlabel $1$ [t] at 197 50
\endlabellist
\centering
 \includegraphics{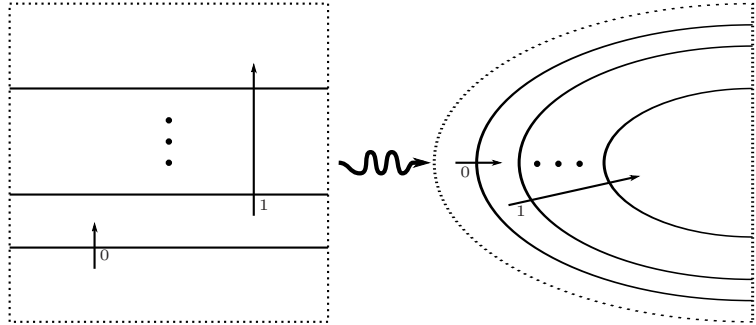}
 \caption{\label{F:Folds1A} The first Morse $2$--function, $G_1$, on
   the $0$-- and $1$--handles of $X$.} 
\end{figure}

Now note that $\partial X_1 = \sharp^{i_1} (S^1 \times S^2)$ sits over
the right edge of the half disk in Figure~\ref{F:Folds1A} and that the
vertical Morse function on $\partial X_1$, i.e. $z \circ G_1
|_{\partial X_1}$ is the standard Morse function with $i_1$ index $1$
critical points and $i_1$ index $2$ critical points, inducing the
standard genus $i_1$ splitting of $\partial X_1$, with Heegaard
surface $F$.  

Consider the framed attaching link $L \subset \partial X_1$ for the
$2$--handles of $X$. Generically $L$ will be disjoint in $\partial
X_1$ from the ascending $1$--manifolds of the index $2$ critical
points of $z \circ G_1 |_{\partial X_1}$ as well as the descending
$1$--manifolds of the index $1$ critical points. Thus $L$ can be
projected onto the Heegaard surface $F$ along gradient flow lines
to give an immersed curve $\overline{L}$ in $F$ with at worst
double points. By adding kinks if necessary, we can assume that the
handle framing of $L$ agrees with the ``blackboard framing'' coming
from $\overline{L} \subset F$. Then by stabilizing this Heegaard
splitting once for each crossing of $\overline{L}$, we can resolve
these crossings and get $L$ to lie in the Heegaard surface with
framing coming from the surface. This process translates into an
extension of the thus-far constructed $G_1$ from $X_1$ to $X_1 \cup
([0,1] \times \partial X_1)$ with fold locus as in
Figure~\ref{F:Folds1B}, with one cusp for each stabilization. In other
words, the sequence of stabilizations translates into a homotopy $g_t$
from $g_0$, the standard Morse function on $\sharp^{i_1} (S^1 \times
S^2)$, to $g_1$, the stabilized Morse function. This homotopy then
becomes a Morse $2$--function on the collar $[0,1] \times \partial
X_1$. 
\begin{figure}
\labellist
\tiny\hair 2pt
\pinlabel $0$ [tr] at 12 65
\pinlabel $1$ [t] at 30 50
\pinlabel $1$ [tl] at 124 37
\pinlabel $2$ [tl] at 124 76
\endlabellist
\centering
 \includegraphics{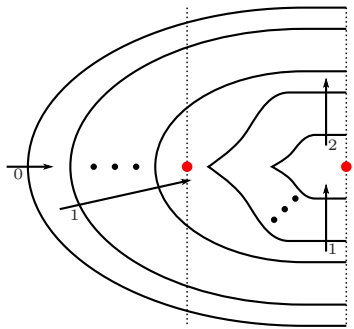}
 \caption{\label{F:Folds1B} $G_1$ extended to a collar on $\partial
   X_1$. In the two vertical slices shown, both diffeomorphic to
   $\sharp^n (S^1 \times S^2)$, the Heegaard surface sits over the
   highlighted red points. The framed attaching link $L$ for the
   $2$--handles of $X$ lies in the Heegaard surface for the right-most
   Morse function, i.e. over the right-most red point, with framing
   coming from the surface.} 
\end{figure}

Now let $F$ refer to the stabilized Heegaard surface, in which
$L$ lies. Attaching a $4$--dimensional $2$--handle to $X_1$ along a
component $K$ of $L$ is the same as attaching $I$ times a
$3$--dimensional $2$--handle to $X_1$ along $I \times K \subset I
\times F \subset \partial X_1$. In Figure~\ref{F:Folds1C} we show
the resulting Morse $2$--function at the left, where the handle sits
over a vertical rectangle. Next we bend this rectangle to make the
image again a half-disk. Finally, noting that the vertical Morse
function at the right edge now has an index $2$ critical value below
an index $1$ critical value, we switch these values to get the Morse
$2$--function at the right side of Figure~\ref{F:Folds1C}.  
\begin{figure}
\labellist
\tiny\hair 2pt
\pinlabel $0$ [tl] at 64 12
\pinlabel $1$ [l] at 58 27
\pinlabel $3$ [bl] at 64 149
\pinlabel $2$ [l] at 59 132
\pinlabel $2$ [tr] at 78 77
\pinlabel $0$ [tl] at 173 12
\pinlabel $1$ [l] at 167 27
\pinlabel $3$ [bl] at 173 149
\pinlabel $2$ [l] at 168 132
\pinlabel $2$ [br] at 177 62
\pinlabel $1$ [tr] at 178 98
\pinlabel $0$ [tr] at 270 12
\pinlabel $1$ [br] at 278 58
\pinlabel $3$ [br] at 270 149
\pinlabel $2$ [tr] at 278 103
\endlabellist
\centering
 \includegraphics{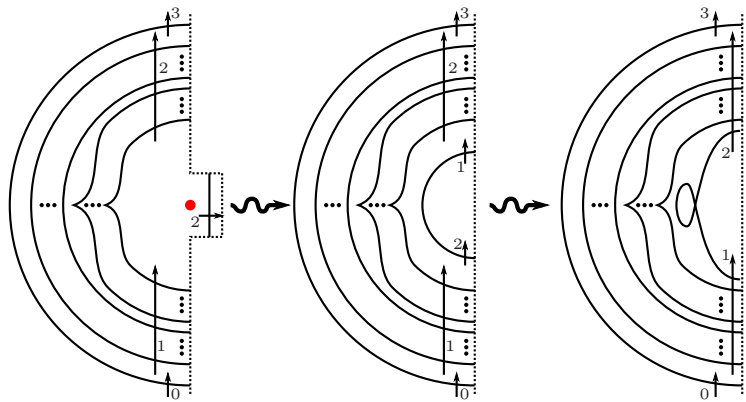}
 \caption{\label{F:Folds1C} $G_1$ after attaching a $4$--dimensional
   $2$--handle.} 
\end{figure}

Note that everything in the preceding paragraph happened in a
neighborhood of $K$, so that the rest of $L$ still lies in the middle
Heegaard surface for the Morse function at the right edge of the final
diagram in Figure~\ref{F:Folds1C}. Thus we can attach each
$4$--dimensional $2$--handle this way to get the Morse $2$--function
at the left side of Figure~\ref{F:Folds1D}. Each $2$--handle of $X$
corresponds to a kink in the image of the folds, i.e. a smoothly
immersed arc with a single transverse double point. Repeating our
construction for $X_1$ with the union of the $3$-- and $4$--handles,
we construct the Morse $2$--function at the right side of
Figure~\ref{F:Folds1D}. The two halves give vertical Morse functions
on the boundary of the union of the $3$-- and $4$--handles, which are
related by some Cerf graphic. Putting this Cerf graphic in between the
two parts of Figure~\ref{F:Folds1D} gives us $G_1$ as in
Figure~\ref{F:Folds1}. 
\begin{figure}
\labellist
\tiny\hair 2pt
\pinlabel $0$ [tr] at 83 11
\pinlabel $1$ [br] at 96 66
\pinlabel $3$ [br] at 82 149
\pinlabel $2$ [tr] at 95 95
\pinlabel $0$ [tl] at 130 12
\pinlabel $1$ [r] at 139 40
\pinlabel $3$ [bl] at 130 148
\pinlabel $2$ [r] at 138 120
\endlabellist
\centering
 \includegraphics{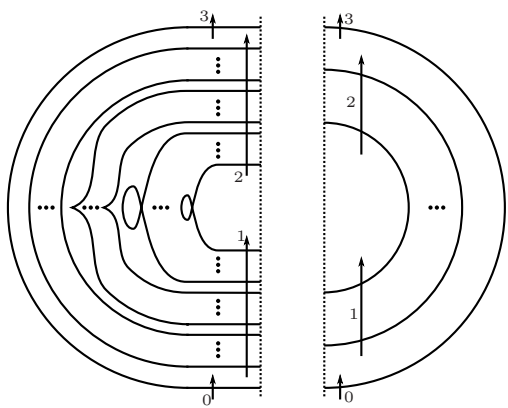}
 \caption{\label{F:Folds1D} Two halves of $G_1$: the $0$--, $1$-- and
   $2$--handles on the left and the $3$-- and $4$--handles on the
   right. Connecting them with a Cerf graphic gives
   Figure~\ref{F:Folds1}.} 
\end{figure}

To get to Figure~\ref{F:Folds2}, first we take the Cerf graphic
section of Figure~\ref{F:Folds1} and pull the births (left-cusps) to
the left of the Cerf graphic and the deaths (right-cusps) to the
right, and then pull all index $1$ critical points below all index $2$
critical points. Then the left-cusps can be pulled further left, past
the kinks which correspond to $4$--dimensional $2$--handles, because
the $4$--dimensional $2$--handle attachments are independent of the
$3$--dimensional stabilizations corresponding to the cusps. This is
shown in Figure~\ref{F:Folds2A}. 
\begin{figure}
\labellist
\tiny\hair 2pt
\pinlabel $1$ at 10 190
\pinlabel $2$ at 190 190
\pinlabel $3$ at 10 90
\pinlabel $4$ at 190 90
\endlabellist
\centering
 \includegraphics{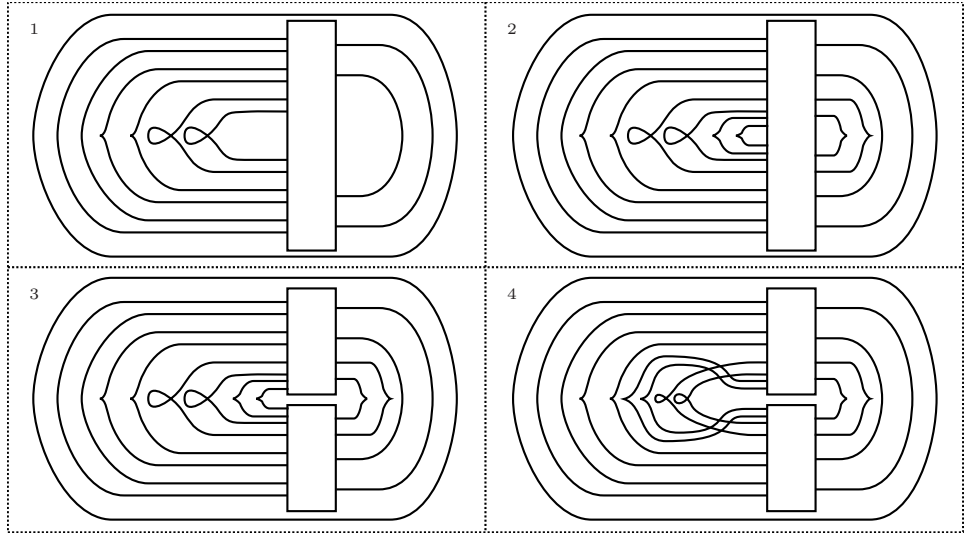}
 \caption{\label{F:Folds2A} Pulling cusps out of the Cerf
   graphic. Here we suppress the ``three dots'' notation as well as
   the indices of the folds, as these are understood from earlier
   figures.} 
\end{figure}

Next we homotope the kinks into pairs of cusps as in
Figure~\ref{F:Folds2B}. The first step of Figure~\ref{F:Folds2B}
introduces a swallowtail at the vertical tangency of the kink; this
move has been discussed extensively elsewhere~\cite{lekili} and is a
standard singularity that occurs in a homotopy between homotopies
between Morse functions. The second step moves an arc of index $1$
critical points in a homotopy (Cerf graphic) below an arc of index $2$
critical points. This is also standard and is possible because the
descending manifold for the index $1$ point remains disjoint from the
ascending manifold for the index $2$ point throughout the
homotopy. (Equivalently, in homotopies between Morse functions we
never expect $1$--handles to slide over $2$--handles.) 
\begin{figure}
\labellist
\tiny\hair 2pt
\pinlabel $2$ [br] at 8 32
\pinlabel $1$ [br] at 100 32
\pinlabel $1$ [br] at 188 32
\pinlabel $1$ [tr] at 66 8
\pinlabel $2$ [tr] at 65 56
\pinlabel $1$ [tr] at 153 8
\pinlabel $2$ [tr] at 152 56
\pinlabel $1$ [tr] at 234 8
\pinlabel $2$ [tr] at 233 56
\pinlabel $1$ [tr] at 121 42
\pinlabel $2$ [tr] at 120 20
\endlabellist
\centering
 \includegraphics{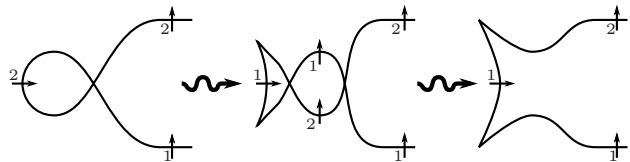}
 \caption{\label{F:Folds2B} Turning kinks into pairs of cusps.}
\end{figure}

Finally, Figure~\ref{F:Folds4A} shows how to add folds and cusps to a
Morse $2$--function as in Figure~\ref{F:Folds4} so as to increase the
number of folds without cusps in one of the three sectors. Here we are introducing an eye, as in Figure~\ref{F:Eye}, modified by a slight isotopy. Note that the transition from the second to the third diagram in the figure is not essential, but only serves to put the resulting diagram in the form of Figure~\ref{F:Folds4}. Depending on how we orient the new eye with
respect to the trisection of $\R^2$, we either add the fold without
cusps to $\R^2_1$, $\R^2_2$, or $\R^2_3$.
\begin{figure}
\labellist
\tiny\hair 2pt
\endlabellist
\centering
 \includegraphics{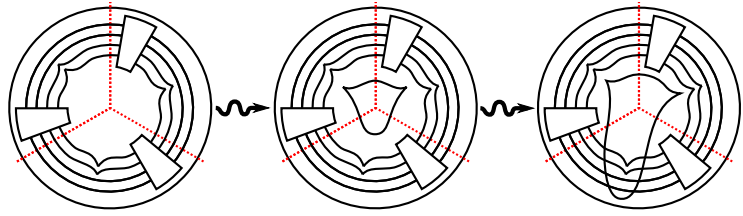}
 \caption{\label{F:Folds4A} Adding an extra fold without cusps in one sector; again we suppress the ``three dots'' notation and the fold indices.} 
\end{figure}

(Note that if we do this
operation three times, once for each sector, we increase $k$ by $1$ and $g$ by $3$; this is precisely a stabilization of the trisection, as will be shown in Section~\ref{S:Uniqueness}.) 

Now we need to show that, having put our Morse $2$--function finally
into the form of Figure~\ref{F:Folds4}, with $k$ folds in each sector
without cusps and $g-k$ folds with cusps, then for each $i$,
$G^{-1}(\R^2_i) = X_i \cong \natural^k (S^1 \times B^3)$. However, we
have already seen this: Each sector, ignoring the Cerf graphic block,
looks just like Figure~\ref{F:Folds1B}, which we already know is
$\natural^k (S^1 \times B^3)$ with a $(g-k)$--times stabilized
standard Heegaard splitting on the boundary. The Cerf graphic block
connecting one sector to another is a product which does not interfere
with the Heegaard splitting.

\end{proof}

\section{Trisections and handle decompositions}

The techniques of the previous section lead to a relationship between trisections and handle decompositions equipped with certain extra data. We will use this relationship both to provide an alternate proof of Theorem~\ref{T:Existence} and to prove Theorem~\ref{T:Uniqueness}.

By a {\em system of compressing disks} for a $3$--dimensional handlebody $H$ of genus $g$, we mean a collection of properly embedded disks $D_1, \ldots, D_g \subset H$ such that cutting $H$ open along $D_1 \cup \ldots \cup D_g$ yields a $3$--ball.

\begin{lemma} \label{L:Trisection2Handles}
 If $X = X_1 \cup X_2 \cup X_3$ is a trisection of a $4$--manifold $X$, then there is a handle decomposition of $X$ as in Theorem~\ref{T:Existence} satisfying the following properties:
 \begin{enumerate}
  \item $X_1$ is the union of the $0$-- and $1$--handles.
  \item Considering the Heegaard splitting $\partial X_1 = H_{12} \cup H_{31}$ with Heegaard surface $F$, the attaching link $L$ for the $2$--handles lies in the interior of $H_{12}$.
  \item The framed attaching link $L=K_1 \cup \ldots \cup K_{g-k}$ is isotopic in $H_{12}$ to a framed link $L'=K'_1 \cup \ldots \cup K'_{g-k} \subset F$, with framings equal to the framings induced by $F$.
  \item There is a system of compressing disks $D_1, \ldots, D_g$ for $H_{12}$ such that the curves $K'_1, \ldots, K'_{g-k}$ are geometrically dual in $F$ to the curves $\partial D_1, \ldots, \partial D_{g-k}$. In other words, each $K'_j$ intersects $\partial D_j$ transversely once and is disjoint from all other $\partial D_i$'s.
  \item  There is a tubular neighborhood $N=[-\epsilon,\epsilon] \times H_{12}$ of $H_{12}$ with $[-\epsilon,0] \times H_{12} = N \cap X_1$, such that $X_2$ is the union of $[0,\epsilon] \times H_{12}$ with the $2$--handles.
 \end{enumerate}
\end{lemma}

\begin{proof}
 Each sector of the trisection of $X$ is diffeomorphic to $\natural^k
(S^1 \times B^3)$ with a genus $g$ splitting of its boundary.  Thus it
has a standard Morse $2$--function onto a wedge in $R^2$, see Figure~\ref{F:Folds4}.
Two sectors meet at $X_i \cap X_{i+1} = \natural^k (S^1 \times B^2)$,
and the two Morse $2$--functions on the two sectors give two Morse
functions on the intersection $X_i \cap X_{i+1}$.  The two Morse
functions are homotopic and thus give a Cerf diagram which can be
inserted into the little wedges in Figure~\ref{F:Folds4}. In the existence proof from the previous section we avoided cusps in the Cerf graphic boxes, but at this point we do not care; any Cerf graphic will do.

An isotopy of $\R^2$ makes the picture look like Figure~\ref{F:hTrisection}. Now projection to the horizontal axis gives a Morse function in which the vertical tangencies become Morse critical points. $X_1$, to the left of the vertical red line, is clearly the union of the $0$-- and $1$--handles. $X_2$, between the legs of the red letter {\em h} is then a handlebody $H_{12}$, cross $I$, with $g-k$ $2$--handles attached. And $X_3$ is obviously what remains. 
\begin{figure}
\labellist
\tiny\hair 2pt
\pinlabel $X_1$ at 135 108
\pinlabel $X_2$ at 178 90
\pinlabel $X_3$ at 230 103
\endlabellist
\centering
 \includegraphics{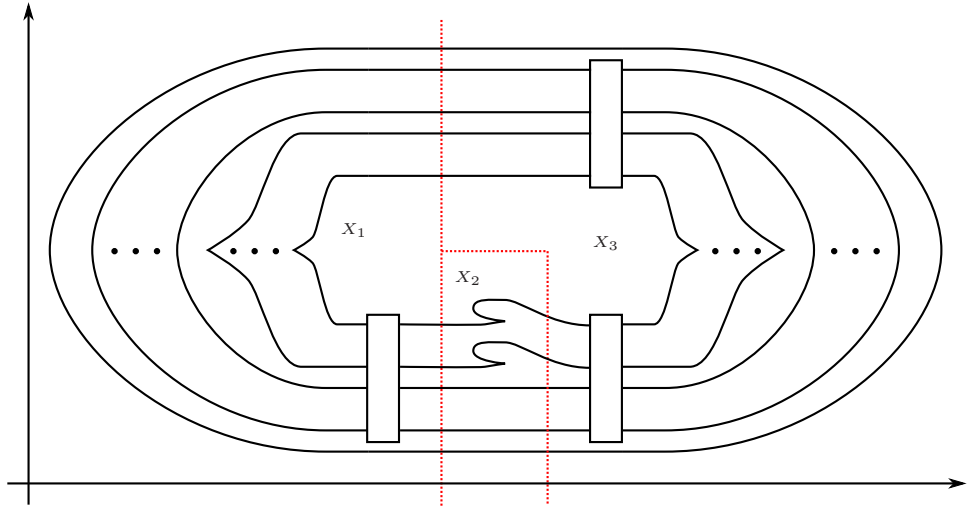}
 \caption{\label{F:hTrisection} Extracting a handle decomposition from a trisection.} 
\end{figure}

We only need to show now that the attaching link for the $2$--handles is as advertised. This can be seen from the fact that the attaching circle for each $2$--handle, between the legs of the {\em h}, is one of a dual pair of curves on the fiber near a cusp. The other curve in the dual pair is the attaching curve for the fold that cuts across $H_{12}$ and gives one of the compressing disks for this handlebody. This is illustrated in Figure~\ref{F:hTrisectionZoom}, which shows a zoomed in region of Figure~\ref{F:hTrisection}. The fiber over a specific point is drawn as a once punctured torus; this is just part of the fiber, but the rest of the fiber does not play a role in this local picture. The attaching circles for the two folds are drawn as green and blue circles on the fiber. This is just the usual picture of the fiber between the two arms of a cusp, with attaching circles being geometrically dual. Here, however, we reinterpret this picture to see the blue circle as the boundary of a compressing disk for the handlebody lying over the vertical dotted red line, and to see the green circle as the attaching circle for the $4$--dimensional $2$--handle coming from the vertical tangency in the fold.
\begin{figure}
\labellist
\tiny\hair 2pt
\endlabellist
\centering
 \includegraphics{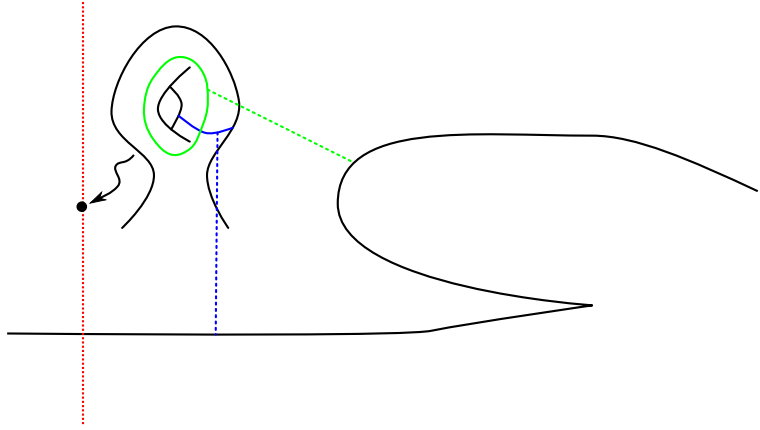}
 \caption{\label{F:hTrisectionZoom} Zooming in on a region of Figure~\ref{F:hTrisection}.} 
\end{figure}

\end{proof}

\begin{lemma} \label{L:Handles2Trisection}
 Consider a handle decomposition of a $4$--manifold $X^4$ with one $0$--handle, $k$ $1$--handles, $g-k$ $2$--handles, $k$ $3$--handles and one $4$--handle. Let $X_1$ be the union of the $0$-- and $1$--handles. Suppose there is a genus $g$ Heegaard splitting $\partial X_1 = H_{12} \cup H_{31}$ of $\partial X_1$ satisfying the following properties in relation to the framed attaching link $L$ for the $2$--handles:
 \begin{enumerate}
  \item $L$ lies in the interior of $H_{12}$.
  \item $L$ is isotopic in $H_{12}$ to a framed link $L' \subset F$, with framing equal to the framing induced by $F$.
  \item There is a system of compressing disks $D_1, \ldots, D_g$ for $H_{12}$ such that the $g-k$ components of $L'$ are, respectively, geometrically dual in $F$ to the curves $\partial D_1, \ldots, \partial D_{g-k}$.
 \end{enumerate}
 Let $N=[-\epsilon,\epsilon] \times H_{12}$ be a small tubular neighborhood of $H_{12}$ with $[-\epsilon,0] \times H_{12} = N \cap X_1$, which the $2$--handles intersect as $[0,\epsilon] \times \nu_L$, where $\nu_L$ is a tubular neighborhood of $L$ in $H_{12}$. Declare $X_2$ to be the union of $[0,\epsilon] \times H_{12}$ with the $2$--handles, and declare $X_3$ to be what remains (the closure of $X \setminus (X_1 \cup X_2)$). Then $X=X_1 \cup X_2 \cup X_3$ is a trisection.
\end{lemma}

\begin{proof}
 Almost everything we need for $X_1 \cup X_2 \cup X_3$ to be a trisection is immediate:
 \begin{enumerate}
  \item $X_1$ and $X_3$ are both diffeomorphic to $\natural^k (S^1 \times B^3)$.
  \item $H_{31} = X_3 \cap X_1$ and $H_{12} = X_1 \cap X_2$ are genus $g$ handlebodies.
  \item $F = X_1 \cap X_2 \cap X_3$ is a genus $g$ surface.
 \end{enumerate}
 It remains to verify that $X_2 \cong \natural^k (S^1 \times B^3)$ and that $H_{23} = X_2 \cap X_3$ is a genus $g$ handlebody.
 
 In fact $X_2$ is built by attaching $g-k$ $2$--handles to $X_{12} \cong \natural^k (S^1 \times B^2)$ along $g-k$ copies of $S^1 \times \{0\} \subset S^1 \times B^2$ in the first $g-k$ $S^1 \times B^3$ summands. Thus the $2$--handles ``cancel'' $g-k$ of the $S^1 \times B^3$'s, giving both desired results immediately.
\end{proof}

Using Lemma~\ref{L:Handles2Trisection}, we now present a proof of the existence of trisections in the spirit of~\cite{OzsSzTriangles}:
\begin{proof}[Proof of Theorem~\ref{T:Existence}, Existence]
  Start with a handle decomposition of $X^4$ with one $0$--handle, $k_1$ $1$--handles, $k_2$ $2$--handles, $k_3$ $3$--handles and one $4$--handle. Add cancelling $1$--$2$ and $2$--$3$ pairs if necessary so as to arrange that $k_1=k_3$. Let $X_1$ be the union of the $0$--handle and the $1$--handles. Note that $\partial X_1$ is a connected sum of $k_1$ copies of $S^1 \times S^2$. Let $L \subset \partial X_1$ be the framed attaching link for the $2$--handles. 

  Consider the genus $k_1$ Heegaard splitting of $\partial X_1$ as $\partial X_1 = H_{12} \cup H_{31}$ with $F=H_{12} \cap H_{31}$. (We will soon be stabilizing this Heegaard splitting, but after each stabilization we will use the same names for the surface and the handlebodies.) The attaching link $L \subset \partial X_1$ can be projected onto the Heegaard surface $F$ with transverse double points (crossings), so that the handle framing is the surface framing. (Add kinks to get the framing right.) Make sure that each component has at least one crossing using Reidemeister 2 moves if necessary. Let $c$ be the number of crossings in this projection. 

  If $c \leq k_2$ then we are almost done. Stabilize the Heegaard splitting exactly $k_2$ times, with $c$ of these stabilizations occuring at the crossings. Then $L$ can be isotoped so as to resolve all the crossings by sending the over strand at each crossing over the new $S^1 \times S^1$ summand in $F$ coming from the stabilization at that crossing. Now we have a genus $g=k_1+k_2$ Heegaard splitting. Letting $k=k_1$ and $g=k_1+k_2$, and pushing $L$ into the interior of $H_{12}$, we now satisfy the hypotheses of Lemma~\ref{L:Handles2Trisection} and apply that lemma to produce our trisection. (We get duality to a system of meridians as follows: Each component $K$ of $L$ goes over at least one stabilization which no other components go over, and therefore is the unique component intersect the meridian for that stabilization. For every other meridian which $K$ intersects, slide that meridian's compressing disk over the compressing disk corresponding to the stabilization singled out in the preceding sentence.)

  If $c > k_2$ then add $c-k_2$ cancelling $1$--$2$ pairs and $c-k_2$ cancelling $2$--$3$ pairs to the original handle decomposition of $X$. Now we have $k'_1 = k_1+c-k_2$ $1$--handles, and the same number of $3$--handles, as well as $k'_2 = 2c-k_2$ $2$--handles. We consider the new $X'_1 = X_1 \natural^{c-k_2} S^1 \times B^3$ with the natural genus $k'_1$ Heegaard splitting $\partial X'_1 = H'_{12} \cup H'_{31}$ with $F' = H'_{12} \cap H'_{31}$. The original attaching link $L$ still projects onto $F'$ in the same way, with the same crossings, since $F'$ is naturally $F \sharp^{c-k_2} S^1 \times S^1$. 

  However, we also have $2(c-k_2)$ new $2$--handles. Half of these, coming from the $1$--$2$ pairs, are attached along the meridians of the $c-k_2$ new $S^1 \times S^1$ summands in $F'$ and thus immediately satisfy the conditions in Lemma~\ref{L:Handles2Trisection}. The other half, coming from the $2$--$3$ pairs, are attached along $0$--framed unknots, which project onto $F'$ as circles bounding disks in $F'$.

  Now stabilize the new Heegaard splitting $2c-k_2$ times: The first $c$ of these stabilizations should happen at the crossings of $L$, allowing us to resolve crossings as before. The other $c-k_2$ of the stabilizations should occur next to the $c-k_2$ $0$--framed unknots. Then each of these unknots is isotoped to go over the new $S^1 \times S^1$ summand coming from the adjacent stabilization. Now the entire attaching link satisfies the hypotheses of Lemma~\ref{L:Handles2Trisection}. The new genus of the stabilized Heegaard splitting of $\partial X'_1$ is $g' = k'_1 + 2c - k_2$. To conclude the theorem by applying Lemma~\ref{L:Handles2Trisection} we need that $k'_2 = g' - k'_1$, and this is precisely what we have arranged. 
\end{proof}

\section{Uniqueness} \label{S:Uniqueness}

We first prove that the stabilization operation of Definition~\ref{D:Stabilization} really does produce a new trisection. This can be done directly, but instead we will do so by showing that, from a Morse $2$--function point of view, this stabilization corresponds to adding three eyes at the center of a trisected Morse $2$--function. After that we can proceed with the proof of uniqueness.

\begin{proof}[Proof of Lemma~\ref{L:Stabilization}]
 We are given a trisection $(X_1,X_2,X_3)$ of $X$, with handlebodies $H_{ij} = X_i \cap X_j$, properly embedded arcs $A_{ij} \subset H_{ij}$, and regular neighborhoods of these arcs $N_{ij} \subset X$.

 As we will see at the beginning of the proof of Theorem~\ref{T:Uniqueness}, it is easy to construct a Morse $2$--function as in Figure~\ref{F:Folds4} which recovers this trisection. We claim that adding three eyes arranged as in Figure~\ref{F:StabilizedMorse2F} modifies each sector $X_i$ exactly as in Definition~\ref{D:Stabilization}, and since the new Morse $2$--function again gives a trisection, then stabilization as defined in Definition~\ref{D:Stabilization} produces a trisection.

We see that the claim is true one eye at a time. Each time we add an eye, first add it away from the center straddling the intersection of two sectors, such as $H_{31}$, as on the left in Figure~\ref{F:TrisectedEye}. We will then pull the lower fold across the central fiber to achieve the right-hand diagram in Figure~\ref{F:TrisectedEye}. Up to isotopy, moving from the left to the right in this figure is the same as not moving the eye, but instead enlarging the lower sector $X_2$ by attaching the inverse image of the green region labelled $N$. This inverse image is in fact a $1$--handle cobordism attached to $X_2$, since this fold is an index $1$ fold going in towards the middle of the eye. Furthermore, the $1$--handle is cancelled by a $2$--handle immediately above it. The $1$--handle and $2$--handle are actually $I$ cross $3$--dimensional $1$-- and $2$--handles, respectively, and thus we see that we have simply removed a neighborhood of an arc in $H_{31}$ from both $X_1$ and $X_3$ and added it to $X_2$. Repeat this for each of the three eyes.

\end{proof}
\begin{figure}
\centering
 \includegraphics[width=1.5in]{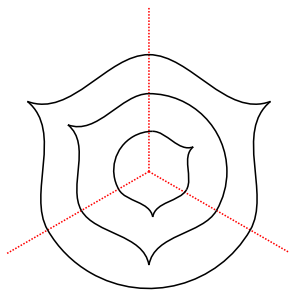}
 \caption{\label{F:StabilizedMorse2F} Stabilizing a Morse $2$--function by adding three ``eyes''.} 
\end{figure}

\begin{figure}
\labellist
\tiny\hair 2pt
\pinlabel $X_2$ at 50 30
\pinlabel $N$ at 158 59
\pinlabel $H_{31}$ [l] at 49 102
\pinlabel $X_2'$ at 268 30
\endlabellist
\centering
 \includegraphics{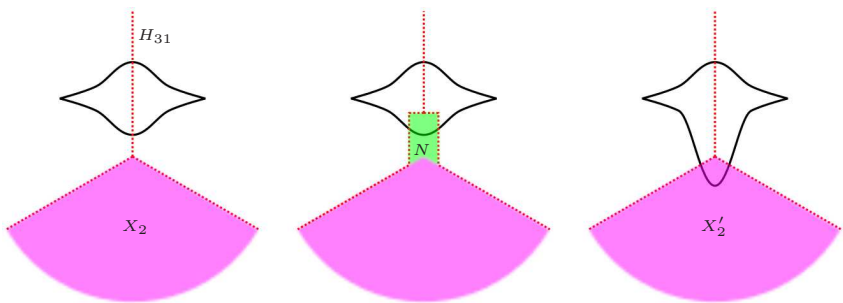}
 \caption{\label{F:TrisectedEye} Adding one eye to a trisected Morse $2$--function.} 
\end{figure}

\begin{proof}[Proof of Uniqueness, Theorem~\ref{T:Uniqueness}]

Consider two trisections of the same $4$--manifold: $X^4 = X_1 \cup X_2 \cup X_3 = X'_1 \cup X'_2 \cup X'_3$. Apply Lemma~\ref{L:Trisection2Handles} to each trisection to get two handle decompositions $D$ and $D'$ of $X$, respectively, with corresponding Heegaard splittings of $\partial X_1$, with attaching links $L$ and $L'$ behaving as in Lemma~\ref{L:Trisection2Handles}. Cerf theory tells us that we can get from $D$ to $D'$ by the following operations:
\begin{enumerate}
 \item Add cancelling $1$--$2$ and $2$--$3$ pairs to both $D$ and $D'$.
 \item Slide $1$--handles over $1$--handles, $2$--handles over $2$--handles and $3$--handles over $3$--handles.
 \item Isotope the handles and their attaching maps without sliding over any handles.
\end{enumerate}

From the description of trisection stabilization in the proof of Lemma~\ref{L:Stabilization} above, we can see that trisection stabilization adds both a $1$--$2$ pair and a $2$--$3$ pair to an associated handle decomposition. Thus, after arranging that we add the same number of $1$--$2$ pairs as $2$--$3$ pairs, we can stabilize the two original trisections to take care of the first operation above.

Clearly sliding $1$--handles over $1$--handles and $3$--handles over $3$--handles, as well as isotoping $1$--handles and $3$--handles without handle slides, does not change the associated trisection.

Thus we are left to investigate the effect of $2$--handle slides and $2$--handle isotopies.

Suppose that we wish to perform a single $2$--handle slide to the handle decomposition $D$.
Associated to the trisection $T$ which gives rise to $D$ we have a Heegaard splitting $H_{12} \cup H_{31}$ for $\partial X_1$, with the attaching link $L$ for the $2$--handles of $D$ lying in $H_{12}$. Isotope $L$ into $\partial H_{12} = F$ so that the components of $L$ are dual to the $g-k$ curves in a system of $g$ meridinal curves (boundaries of compressing disks), as in Lemma~\ref{L:Trisection2Handles}. The handle slide involves a framed arc connecting two components $K_1$ and $K_2$ of $L$. This arc can be projected (following the flow of a Morse function of $\partial X_1$ for the given Heegaard splitting) onto $F$, but with crossings. We can arrange for its framing to agree with the surface framing with kinks, as usual. We want to avoid self-crossings as well as crossings between the arc and $L$ and between the arc and the system of meridinal curves.

Stabilizing the Heegaard splitting, however, allows us to resolve the crossings. In other words, we get a new Heegaard splitting $\partial X_1 = H_{12}' \cup H_{31}'$ obtained from $H_{12} \cup H_{31}$ by Heegaard splitting stabilizations and isotopy such that $L$ and the band lie in $\partial H_{12}' = F'$, still maintaining the property that the components of $L$ are dual to the first $g-k$ meridinal curves in a system of meridinal curves of $H_{12}'$. In addition, the bands are disjoint from these $g-k$ meridinal curves. (Note that we can do this without moving $L$ or the bands, but just by stabilizing and isotoping the Heegaard splitting.) Then sliding one component of $L$ over another along the chosen band maintains this property; we have to change one of the meridinal curves in the system of compressing disks by a handle slide as well.

Again, from the proof of Lemma~\ref{L:Stabilization}, we see that stabilization of the Heegaard splitting of $\partial X_1$ can be achieved by stabilizing the trisection, at the expense of introducing cancelling $1$--$2$ and $2$--$3$ pairs to the associated handle decomposition.

Thus we have shown that, if $D$ and $D'$ are related by handle slides supported in small neighborhoods of arcs in $\partial X_1$, then they are adapted to trisections related by trisection stabilization and isotopy.

Finally, suppose that $D$ and $D'$ are related only by an isotopy of the $2$--handles and their attaching maps, without any handle slides. Then this isotopy extends to an isotopy of $X$ with the result that we can assume that the handle decompositions are identical, and the only difference between the trisections is the Heegaard splitting of $\partial X_1$. 

So we have two Heegaard splittings $\partial X_1 = H_{12} \cup H_{31} = H_{12}' \cup H_{31}'$, respectively, coming from $T$ and $T'$. The fixed attaching link $L$ for the $2$--handles lies in both $H_{12}$ and $H_{12}'$, in both cases satisfying the condition of being dual to meridinal curves.

Note that both $H_{12} \cup H_{31}$ and $H_{12}' \cup H_{31}'$ are genus $g$ Heegaard splittings of $\partial X_1 \cong \natural^k (S^1 \times S^2)$, so that Waldhausen's theorem~\cite{Waldhausen} gives us an isotopy of $\partial X_1$ taking $H_{12}$ to $H_{12}'$. However, this {\em does not} imply that the trisections $T$ and $T'$ are isotopic, because this isotopy will in general move the link $L$. If we can find an isotopy that does not move $L$, then we will be done, but first we will probably need to stabilize.

To see how to do this, construct two Morse functions $f$ and $f'$ on $\partial X_1$ with regular values $a < b$ such that:
\begin{enumerate}
 \item $f$ and $f'$ agree on $f^{-1}(-\infty,a] = f'^{-1}(-\infty,a]$, which is a tubular neighborhood of $L$ (thus each has $g-k$ index $0$ critical points and $g-k$ index $1$ critical points),
 \item $f^{-1}(-\infty,b]=H_{12}$,
 \item $f'^{-1}(-\infty,b]=H_{12}'$,
 \item $f$ and $f'$ have only critical values of index $1$ in $[a,b]$ and critical values of index $2$ and $3$ in $[b,\infty)$.
\end{enumerate}
Now Cerf theory gives us a homotopy $f_t$ from $f_0 = f$ to $f_1 = f'$ which involves $1$--$2$ births and deaths on $f^{-1}(b)$ and otherwise no critical values crossing $b$, and such that $f_t = f = f'$ on $f^{-1}(\infty,a]$. Thus, after stabilizing the Heegaard splittings away from $L$, there is an isotopy fixing $L$ taking the one Heegaard splitting to the other.

Again, the Heegaard splitting stabilizations are achieved by trisection stabilizations.

\end{proof}

\begin{remark}

 Morally it seems that there should be a Morse $2$--function proof of uniqueness that starts with a generic homotopy between two Morse $2$--functions corresponding to two given trisections. Then the proof would homotope this homotopy so as to arrange that the Cerf $2$--graphic in $[0,1] \times \R^2$, a surface of folds with cusps and higher codimension singularities, is in a nice position with respect to the standard trisection of $[0,1] \times \R^2$. This surface of folds is, however, not trivial to work with. A good model might be the method of braid foliations used by Birman and Menasco to prove Markov's theorem in~\cite{BirmanMenasco}.
\end{remark}

\section{The relative case}

When $\p X \neq \emptyset$, we should define a trisection as the kind of subdivision of $X$ which naturally arises from a Morse $2$--function $G:X \to B^2$ where $B^2$ is trisected as in Figure~\ref{F:BasicTrisection}, the locus of critical values behaves well with respect to this trisection of $B^2$, and the trisection of $X$ is just $G^{-1}$ of the three sectors of $B^2$. ``Behaving well'' should mean that the folds all have index $1$ when transversely oriented towards the center of $B^2$, that the only tangencies to rays of $B^2$ are the cusps, that there is at most one cusp per fold in each sector, and that each sector has the same number of cusps. We now formulate this without mention of a Morse $2$--function.

First, when $M^3$ has a boundary $\p M$, then a Heegaard splitting is a splitting into compression bodies rather than solid handlebodies. Traditionally, a compression body is the result of attaching $n \leq k$ $3$--dimensional $2$--handles to $\{1\} \times F_k \subset [0,1] \times F_k$ so as to get a cobordism from $F_k$ to $F_{k-n}$, where $F_k$ is a closed surface of genus $k$. In fact, we can even consider the case where $F$ is a compact surface $F_{k,b}$ of genus $k$ with $b \geq 0$ boundary components, in which case we get a cobordism to $F_{(k-n),b}$. Note that the diffeomorphism type of such a cobordism is completely determined by $k$, $b$ and $n$; let $C_{k,b,n}$ denote a standard model for this compression body. To summarize, both ends of $C_{k,b,n}$ are surfaces with $b$ boundary components, the higher genus end has genus $k$ and there are $n$ compression disks yielding a lower genus end with genus $k-n$.

Now consider $Z_{k,b,n}=[0,1] \times C_{k,b,n}$. Part of $\partial Z_{k,b,n}$ is $Y_{k,b,n} = (\{0\} \times C_{k,b,n}) \cup ([0,1] \times F_{k,b}) \cup (\{1\} \times C_{k,b,n})$, which has a natural genus $k$ Heegaard splitting into two compression bodies $Y_{k,b,n}^+ = ([1/2,1] \times F_{k,b}) \cup (\{1\} \times C_{k,b,n})$ and $Y_{k,b,n}^- = (\{0\} \times C_{k,b,n}) \cup [0,1/2] \times F_{k,b})$. Finally, given any $g \geq k$, let $Y_{k,b,n}=Y_{k,b,n,g}^+ \cup Y_{k,b,n,g}^-$ be the genus $g$ Heegaard splitting obtained from the natural genus $k$ splitting by stabilizing $g-k$ times.

\begin{definition} \label{D:RelTrisection}
 A trisection of a $4$--manifold $X$ with boundary is a splitting $X=X_1 \cup X_2 \cup X_3$ and integers $0 \leq k,b,n,g$ with $n \leq k \leq g$ such that each $X_i$ is diffeomorphic to $Z_{k,b,n}$ via a diffeomorphism $\phi_i : X_i \to Z_{k,b,n}$ for which
 \[ \phi_i(X_i \cap X_{i+1}) = Y_{k,b,n,g}^+\]
 and
 \[ \phi_i(X_i \cap X_{i-1}) = Y_{k,b,n,g}^-\]
\end{definition}

We leave the proof of the following to the reader:
\begin{lemma}
 A trisection of a $4$--manifold $X$ with non-empty boundary restricts to the boundary $M^3 = \partial X$ as either a fibration over $S^1$ (when $b = 0$) or an open book decomposition (when $b \neq 0$). In the first case, $X_i \cap \partial X$ is the inverse image under the fibration of $[2\pi i/3, 2\pi (i+1)/3] \subset S^1$. In the second case, $X_i \cap \partial X$ is the union of this inverse image and the binding.
\end{lemma}

\begin{remark}
Lefschetz fibrations over $B^2$ can be perturbed to give examples of trisections in this relative setting.  Assume that $f:X^4 \to B^2$ is a bundle with fiber $F_{k,b}$ except for exceptional fibers which have nodes where in local coordinates $(z,w)$ $f$ is given by $f(z,w)= zw$. Lekili showed in \cite{lekili} that the map $f$ could be locally perturbed so that the node is replaced by three $1$--folds in the shape of a hyperbolic triangle, as in Figure~\ref{F:NodeToTriangle}. We need such a triangle to go around the central fiber of our trisection, so we move a cusp up to and past the central fiber.  This ups the genus of the central fiber by one.  Now it is easy to trisect $X$ for the only folds are these triangles.
\begin{figure}
\labellist
\tiny\hair 2pt
\pinlabel {node} [t] at 3 21
\pinlabel $1$ [br] at 96 29
\pinlabel $1$ [bl] at 121 29
\pinlabel $1$ [t] at 109 4
\endlabellist
\centering
 \includegraphics{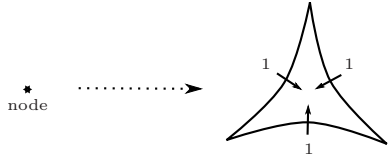}
 \caption{\label{F:NodeToTriangle} Perturbation of a Lefschetz node singularity.}
\end{figure}
\end{remark}

\begin{remark}
 Given two $4$--manifolds $X$ and $X'$, with diffeomorphic boundary, both trisected with $b=0$, and with a gluing map $\partial X \to -\partial X'$ respecting trisections, gluing along the boundary does not imply produce a trisection of the closed manifold $X \cup X'$. However, we naturally have six pieces which fit together like the faces of a cube. From this, the technique described in Section~\ref{S:Examples} for producing a trisection of a bundle over $S^2$ can be generalized to give a natural trisection of $X \cup X'$.
\end{remark}

\begin{theorem}
Given a $4$--manifold $X$ with an open book decomposition or fibration over $S^1$ on $\partial X$, there exists a trisection of $X$ restricting to $\partial X$ as the given fibration or open book.
\end{theorem}

\begin{proof}
 Use the given boundary data to see $X$ as a cobordism from $F \times [0,1]$ to $F \times [0,1]$, where $F$ is either the fiber or the page. Using a handle decomposition of $X$ compatible with this cobordism structure, repeat the second version of the proof of Theorem~\ref{T:Existence}.
\end{proof}

Stabilization of trisections makes sense in the relative case, since it takes place inside a ball in the interior of $X$. 

\begin{theorem}
Any two trisections of a fixed $4$--manifold $X$ which agree on $\partial X$ are isotopic after stabilizations.
\end{theorem}

\begin{proof}
 Again, the proof of Theorem~\ref{T:Uniqueness} works verbatim in this case, once we fix the appropriate cobordism structure on $X$. The key idea is that Cerf theory works perfectly well when we fix behavior on compact subsets.
\end{proof}

\end{document}